\documentclass{amsart}
\usepackage{hyperref}
\usepackage{texdraw}
\usepackage{color}
\usepackage{enumitem}
\usepackage[pdftex]{graphicx}

\makeatletter

\newfont{\gothic}{eufm10 scaled 1100}

\theoremstyle{plain}    
\newtheorem{thm}{Theorem}[section]
\numberwithin{equation}{section} 
\numberwithin{figure}{section} 
\theoremstyle{plain}    
\theoremstyle{plain}    
\newtheorem{conj}[thm]{Conjecture} 
\theoremstyle{plain}    
\theoremstyle{plain}
\newtheorem{lem}[thm]{Lemma} 
\theoremstyle{plain}    
\newtheorem{prop}[thm]{Proposition} 
\theoremstyle{plain}    
\newtheorem{Def}[thm]{Definition} 
\theoremstyle{plain}    
\newtheorem{const}[thm]{Construction}
\theoremstyle{remark}
\newtheorem{rem}[thm]{Remark}
\theoremstyle{remark}    

\theoremstyle{remark}

\usepackage[arrow,matrix,curve,ps]{xy}
\xyoption{dvips}
\CompileMatrices

\makeatother

\begin{document}

\title{Ciliberto-Miranda degenerations of $\mathbb{CP}^2$ blown up in $10$ points}

\date{\today}

\author{Thomas Eckl}




\address{Thomas Eckl, Department of Mathematical Sciences, The University of Liverpool, Mathematical
               Sciences Building, Liverpool, L69 7ZL, England, U.K.}

\email{thomas.eckl@liv.ac.uk}

\urladdr{http://pcwww.liv.ac.uk/~eckl/}

\maketitle

\begin{abstract}
We simplify Ciliberto's and Miranda's method~\cite{CM08} to construct degenerations of
$\mathbb{CP}^2$
blown up in several points yielding lower bounds of the corresponding multi-point Seshadri constants. In particular we exploit 
an asymptotic result of \cite{Eck08} which allows to check the non-specialty of much fewer linear systems on 
$\mathbb{CP}^2$.
We obtain the lower bound $\frac{117}{370}$ for the $10$-point Seshadri constant on $\mathbb{CP}^2$.
\end{abstract}


\pagestyle{myheadings}
\markboth{THOMAS ECKL}{CILIBERTO-MIRANDA DEGENERATIONS}

\setcounter{section}{-1}

\section{Introduction}

\noindent 
\begin{conj}[Nagata, \cite{Nag59}]
Let $p_1, \ldots, p_n$ be $n \geq 10$ points on $\mathbb{CP}^2$ in general position, and let $C$ be an irreducible curve of
degree $d$ on $\mathbb{P}^2$, passing with multiplicity $m_i$ through the point $p_i$. Then
\[ d > \sqrt{n} \sum_{i=1}^n m_i. \]
\end{conj}

\noindent Cast in the language of Seshadri constants, Nagata claimed in effect that
$\frac{1}{\sqrt{n}}$ is the multi-point Seshadri constant of $p_1, \ldots, p_n \in \mathbb{CP}^2$, for the line
bundle $\mathcal{O}_{\mathbb{P}^2}(1)$, or equivalently, that 
\[ H - \sqrt{\frac{1}{n}} \sum_{j=1}^n E_j \]
is a nef $\mathbb{R}$-divisor on $\widetilde{X} = \mathrm{Bl}_n(\mathbb{P}^2)$, the blowup of $\mathbb{P}^2$ in $n$ points,
where $H$ is the pullback of a line in $\mathbb{P}^2$ and $E_j$ are the exceptional divisors over the blown up 
points.

\noindent The best known bounds for the Seshadri constant of $10$ points on $\mathbb{P}^2$ until very recently were 
$\frac{6}{19}$ by Biran~\cite{Bir99} and $\frac{177}{560}$ by Harbourne and Ro\'e~\cite{HR03}. Some months ago, Ciliberto
and Miranda~\cite{CM08} presented a new method to improve these bounds, and obtained $\frac{55}{174}$. 

\noindent Their approach relies on the well known fact that Nagata's conjecture can be deduced from another conjecture on the
dimension of linear systems on $\mathbb{CP}^2$ (see e.g. \cite{CilMir01}):
\begin{conj}[Harbourne-Gimigliano-Hirschowitz \cite{Har85b, Gim87, Hir89}]
Let $p_1, \ldots, p_n$ be $n$ points on $\mathbb{CP}^2$ in general position, and let 
$\pi: X \rightarrow \mathbb{CP}^2$ be the blow up of these $n$ points. Furthermore, call $H$ the divisor class of a line on
$\mathbb{CP}^2$, and denote the exceptional divisor over $p_i$ with $E_i$. Given a degree $d$ and $n$ multiplicities 
$m_1, \ldots, m_n$, the linear system $| d \pi^\ast H - \sum_{i=1}^n m_i E_i |$ has the expected dimension
\[ \max (-1, \frac{d(d+3)}{2} - \sum_{i=1}^n \frac{m_i(m_i+1)}{2}) \]
iff there exists no $(-1)$-curve $C$ on $X$ such that 
\[ C . (d \pi^\ast H - \sum_{i=1}^n m_i E_i) \leq -2. \] 
\end{conj}

\noindent Linear systems on $\mathbb{P}^2$ are often analysed via degenerations: If the degenerated linear system on 
the central fiber of the $\mathbb{P}^2$-degeneration has expected dimension, then nearby fibers inherit this property by
semi-continuity. In~\cite{CilMir98} Ciliberto and Miranda use a degeneration of $\mathbb{P}^2$ into a union of $\mathbb{P}^2$
and the
first Hirzebruch surface $\mathbb{F}_1$ to check the Harbourne-Hirschowitz conjecture in a number of cases. Unfortunately,
for Nagata's conjecture the results do not yield interesting bounds for Seshadri constants. The failure is due to $(-1)$-curves
which intersect the degenerated linear systems negatively.

\noindent In~\cite{CM08} Ciliberto and Miranda observe that the normal bundle of these "bad" $(-1)$-curves is negative. Their
new idea is to flop these curves, possibly after some blow ups, thus removing them from a new degeneration. Iterating
these flops of "bad" curves Ciliberto and Miranda obtain $\frac{55}{174}$ as a lower bound for the Seshadri constant of $10$
points on $\mathbb{P}^2$.

\noindent The main technical difficulty in their calculations is the study of linear systems with small expected dimension.
They require an intricate case-by-case analysis. To avoid this as much as possible this paper uses an approximative
approach to Nagata's conjecture developped in~\cite{Eck08}:
\begin{thm} [\cite{Eck08}] \label{Approx-thm}
Let $p_1, \ldots, p_n$ be $n \geq 10$ points on $\mathbb{CP}^2$ in general position, and let 
$\pi: X \rightarrow \mathbb{CP}^2$ be the blow up of these $n$ points. If $(d_i, m_i)$, $i \in \mathbb{N}$, is a sequence of 
integer pairs, such that the linear system $|d_i \pi^\ast H - m_i \sum_{j=1}^n E_j|$ is non-empty of expected dimension, and 
$\frac{d_i}{m_i} \stackrel{i \rightarrow \infty}{\longrightarrow} \frac{a}{\sqrt{n}}$ then
the $\mathbb{R}$-divisor
\[ \pi^\ast H - \frac{a}{\sqrt{n}} \sum_{j=1}^n E_j \] 
is nef on $X$.
\end{thm}

\noindent A first attempt to apply this method was made in~\cite{Eck08b}, using Dumnicki's reduction algorithm~\cite{Dum07},
but only led to the uninteresting bound $\frac{4}{13}$. In this paper the combination with Ciliberto-Miranda degenerations 
yields the new and up to now best known lower bound $\frac{117}{370} \approx 0.31621\ldots$: compared to 
$\frac{55}{174} \approx 0.31609\ldots$ this is one decimal closer to $\sqrt{10} \approx 0.31622\ldots$.

\noindent Besides the approximative approach there are two other new ingredients in this paper:
\begin{enumerate}[leftmargin=*]
\item We consistently use a non-specialty criterion for line bundles, which generalizes the core of Harbourne's Criterion 
for linear systems on $\mathbb{P}^2$ blown up in several points, in~\cite{Har85}. It also works for Hirzebruch surfaces.
\item Instead of flopping the "bad" $(-1)$-curves we only blow them up until the exceptional divisor is isomorphic to
$\mathbb{P}^1 \times \mathbb{P}^1$. According to the Atiyah flop Ciliberto and Miranda continue blowing down the other
fibering of $\mathbb{P}^1 \times \mathbb{P}^1$ thus really erasing the "bad" curve. But in this way they produce non-normal
components, and to prove non-specialty they must again pull back to the blown up components. 

\noindent To avoid this extra turn we just modify the degenerated line bundle on the blow up such that its intersection with the
"bad" curves vanishes. The central fiber of our degenerations thus contain more components but we still consider this 
procedure to be more transparent. In particular it lead us to discover two degenerations underlying the Fifth Degeneration of
Ciliberto and Miranda, and finally to the new bound $\frac{117}{370}$. 
\end{enumerate}  

\noindent The steps of the method and the statements proving the correctness are presented in Section~\ref{CM-method-sec}.
The calculations leading to the new bound $\frac{117}{370}$ are the contents of Section~\ref{CM-deg-sec}. In the last
paragraphs we show
that the next degeneration is considerably more complicated than the previous ones, and we discuss some conditions which
would guarantee that the algorithm never terminates.

\vspace{0.2cm}

\noindent \textbf{Acknowledgements.} The author wants to thank Joaquim Ro\'e who invited him to the
Workshop on Seshadri Constants, held in Barcelona, November 2008.

\noindent The author also wants to thank Ciro Ciliberto for explaining his and Miranda's method during this workshop.
Obviously
this paper owes a lot to their work and also concentrates on the case of 10 points on $\mathbb{P}^2$. Instead, the algorithm
could be used to
produce lower bounds for Seshadri constants of $11$ or more points on $\mathbb{P}^2$. On the other hand, the apparent 
similarities allow to compare the approach in this paper with that of Ciliberto and Miranda particularly well. 

\vspace{0.2cm}

\noindent \textbf{Notation.} 
We consider smooth complex projective surfaces $X$ and sequences of morphisms 
\[ X_n \stackrel{\pi_n}{\rightarrow} X_{n-1} \stackrel{\pi_{n-1}}{\rightarrow} \ldots \stackrel{\pi_1}{\rightarrow} X_0 := X,  \]
where each $\pi_{i}$ is the blow up of a point $p_i \in X_{i-1}$. We also denote $X_i$ by $X(p_1, \ldots, p_i)$, and set
$\pi := \pi_n \circ \ldots \circ \pi_1$. 

\noindent Note that the $p_1, \ldots, p_n$ are not assumed to be in general position. For example, the point $p_i$ can be
mapped onto the point $p_j$ by the intermediate blow downs, if $i > j$. Then $p_i$ is said to
be infinitely near to $p_j$. Sometimes we emphasize this relation by brackets: $[p_1; p_2, \ldots, p_k]$ means that the points
$p_2, \ldots, p_k$ are infinitely near to $p_1$. Each of the $p_2, \ldots, p_k$ can again be replaced by pairs of infinitely near
points etc. 

\noindent $E_i \subset X_i = X(p_1, \ldots, p_i)$ denotes the exceptional divisor over $p_i$, the divisor 
\[ \mathcal{E}_i := \pi_n^\ast \cdots \pi_{i+1}^\ast E_i \]
denotes the pullback on $X_n$.  

\noindent $(\mathcal{E}_1, \ldots, \mathcal{E}_n)$ is called an \textit{exceptional configuration} on $X_n$. We know
\[ \mathcal{E}_i^2 = -1,\ \mathcal{E}_i.\mathcal{E}_j = 0, i \neq j,\ \pi^\ast \mathcal{F}.\mathcal{E}_i = 0 \]
for all line bundles $\mathcal{F}$ on $X$, and $\mathrm{Pic}(X_n)$ is generated by $\mathrm{Pic}(X)$ and 
$\mathcal{E}_1, \ldots, \mathcal{E}_n$.

\noindent If $X = \mathbb{P}^2$ we set $\mathcal{E}_0 := \pi^\ast \mathcal{O}_{\mathcal{P}^2}(1)$, and 
\[ \mathcal{L}(d;m_1, \ldots, m_n) := d \cdot \mathcal{E}_0 - \sum_{i=1}^{n} m_i \mathcal{E}_i. \]
Sometimes, multiple occurences of the same coefficient $m$ is abbreviated by $m^k$.

\noindent Later on, degree and multiplicities will linearly depend on paramaters $d, m, a$. We do not abbreviate these forms
by introducing new letters thus making the notation of some line bundles quite cumbersome. But we prefer to leave the
dependence of degrees and multiplicities transparent and easy to analyse. 

\noindent Finally, $\mathbb{F}_0 \cong \mathbb{P}^1 \times \mathbb{P}^1$, $\mathbb{F}_1 \cong \mathbb{P}^2(p)$, $\ldots$,
$\mathbb{F}_k$, $\ldots$ denote the Hirzebruch surfaces, with projections $\pi_{\mathbb{F}_k}$ to $\mathbb{P}^1$.
Accepting some ambiguity
$E_k \subset \mathbb{F}_k$ denotes the curve at infinity, with self intersection $-k$, whereas $F_k$ denotes a fiber of the 
$\mathbb{P}^1$-bundle $\mathbb{F}_k$. On $\mathbb{F}_k(p_1, \ldots, p_n)$,
\[ \mathcal{L}_{\mathbb{F}_k}(d_1,d_2;m_1, \ldots, m_n) := d_1 \cdot E_k + d_2 \cdot F_k - \sum_{i=1}^{n} m_i \mathcal{E}_i. \]

\section{The Ciliberto-Miranda method}  \label{CM-method-sec}

\subsection{Degenerations and the Gluing Lemma}

\noindent Degenerations are a well-known tool to study (complete) linear systems.

\begin{prop}  \label{degen-prop}
Let $f: \mathcal{X} \rightarrow \Delta$ be a reduced family of projective complex schemes over the unit disc $\Delta$, and
let $\mathcal{L}$ be a line bundle on $\mathcal{X}$. Let $X_t$ denote the fiber of $\mathcal{X}$ over $t \in \Delta$, and set 
$L_t := \mathcal{L}_{| X_t}$. Then:
\[ h^1(X_0, L_0) = 0 \Longrightarrow h^1(X_t, L_t) = 0, \]
for all $t \in \Delta^\prime \subset \Delta$, a smaller unit disc.
\end{prop} 
\begin{proof}
This is a consequence of upper-semicontinuity of the $h^1$-function on flat families of projective schemes 
\cite[Thm.III.12.8]{Hart:AG}. The
flatness follows because $\mathcal{X}$ is reduced over a 1-dimensional smooth base \cite[Prop.III.9.7]{Hart:AG}. 
\end{proof}

\noindent Using this proposition on a given degeneration requires to calculate $H^1(X_0, L_0)$.
\begin{lem}[Gluing Lemma \cite{CM08}]     \label{glue-lem}
Let $X = \bigcup_{i=1}^{n} V_i$ be a union of projective complex schemes, where the $V_i$ are closed subschemes of $X$. Set
$W_k := \bigcup_{i=1}^{k} V_i$, $k=1, \ldots, n$, and denote by $C_{k-1}$ the scheme-theoretic intersection 
$V_k \cap W_{k-1}$, $k=2, \ldots, n$. 

\noindent Let $L$ be a line bundle on $X$ satisfying
\begin{itemize}
\item[(i)] $H^1(V_i, L_{|V_i}) = 0$ for $i = 1, \ldots, n$, 
\item[(ii)] the difference maps $H^0(W_{k-1}, L_{|W_{k-1}}) \oplus H^0(V_k, L_{|V_k}) \rightarrow H^0(C_{k-1}, L_{|C_{k-1}})$ are 
surjective for all $k = 2, \ldots, n$.
\end{itemize}
Then: $H^1(X,L) = 0$.
\end{lem}
\begin{proof}
This follows inductively from the long exact cohomology sequences obtained of the short exact sequences in the next lemma,
applied on $W_k = W_{k-1} \cup V_k$ and $L_{|W_k}$.
\end{proof}

\begin{lem}
Let $X = V \cup W$ be a projective complex scheme, where $V, W$ are closed subschemes of $X$, $C = V \cap W$ the
scheme-theoretic intersection, and $\mathcal{L}$ an invertible sheaf on $X$.

\noindent Then there exists an exact sequence of coherent sheaves on $X$,
\[ 0 \rightarrow \mathcal{L} \rightarrow \mathcal{L}_{|V} \oplus \mathcal{L}_{|W} \rightarrow \mathcal{L}_{|C} \rightarrow 0, \]
where $\mathcal{L}_{|V} \oplus \mathcal{L}_{|W} \rightarrow \mathcal{L}_{|C}$ is the difference map. 
\end{lem}
\begin{proof}
The exactness of the sequence can be checked on open affine subsets $\mathrm{Spec} A$, on which $\mathcal{L}$ is trivial. If
$I_V, I_W, I_C \subset A$ are the
ideals describing the closed subschemes $V, W, C$ in $\mathrm{Spec} A$, the claim follows from $I_V + I_W = I_C$ and
$I_V \cap I_W = (0)$.
\end{proof}

\begin{rem}   \label{diffmap-rem}
Condition (ii) of the Gluing Lemma~\ref{glue-lem} is already satisfied if 
$H^1(W_{k-1}, L_{|W_{k-1}} \otimes \mathcal{I}_{C_{k-1}/W_{k-1}}) = 0$ or 
$H^1(V_k, L_{|V_k} \otimes \mathcal{I}_{C_{k-1}/V_k}) = 0$. Here, $\mathcal{I}_{C_{k-1}/W_{k-1}}$ resp.\/ 
$\mathcal{I}_{C_{k-1}/V_k}$ are the ideal sheaves of $C_{k-1}$ in $W_{k-1}$ resp.\/ $V_k$. This follows from the long exact
cohomolgy sequence associated to
\[ 0 \rightarrow L_{|W_{k-1}} \otimes \mathcal{I}_{C_{k-1}/W_{k-1}} \rightarrow L_{|W_{k-1}} \rightarrow L_{|C_{k-1}}
      \rightarrow 0 \]
(or the analogue sequence for $V_k$), because then
\[ H^0(W_{k-1}, L_{|W_{k-1}}) \rightarrow H^0(C_{k-1}, L_{|C_{k-1}}) \rightarrow   
   H^1(W_{k-1}, L_{|W_{k-1}} \otimes \mathcal{I}_{C_{k-1}/W_{k-1}}) = 0 \]
is exact (or the analogue sequence for $V_k$). 
\end{rem}

\subsection{Non-special linear systems} \label{non-spec-sec}

\noindent The degenerations of $\mathbb{P}^2$ blown up in 10 points studied later on have a central fiber $X_0$ consisting of
irreducible components isomorphic to $\mathbb{P}^2$ or a Hirzebruch surface $\mathbb{F}_k$ blown up in several points, 
possibly in special position, and intersecting in curves without embedded points. 
Then $L_{|V_k} \otimes \mathcal{I}_{C_{k-1}/V_k} = L_{|V_k} \otimes \mathcal{O}_{V_k}(-C_{k-1})$ is again a line bundle. In
view of Remark~\ref{diffmap-rem} this implies for applying the Gluing Lemma~\ref{glue-lem} that we only need criteria for the
vanishing of $H^1$-groups of line bundles on such surfaces.

\noindent The first vanishing criterion is extracted from the central argument of Harbourne's Criterion discussed afterwards.
\begin{thm}  \label{curvevan-crit}
Let $\mathbb{F}$ be $\mathbb{P}^2$ or a Hirzebruch surface $\mathbb{F}_k$.
Let $X = \mathbb{F}(p_1, \ldots, p_n)$ be a blow up of $\mathbb{F}$ in several points. Let $\mathcal{F}$ be a line bundle
on $X$, and set
\[ l := \min \{ k: \mathcal{F} = \pi_n^\ast \cdots \pi_{k+1}^\ast \mathcal{F}_k,\ \mathcal{F}_k\ \mathrm{line\ bundle\ on\ } 
                                                                                                                             X_k = \mathbb{F}(p_1, \ldots, p_k) \}. \]
Let $C$ be a reduced curve on $X_l$ with irreducible components $C_1, \ldots, C_r$. Assume that
\begin{itemize}
\item[(i)] $\left[ K_{X_l} \otimes \mathcal{O}_{X_l}(C)\right].C_i < \mathcal{F}_l.C_i$ for all $i = 1, \ldots, r$, and
\item[(ii)] $H^1(X_l, \mathcal{F}_l \otimes \mathcal{O}_{X_l}(-C)) = 0$.
\end{itemize}
Then: $H^1(X, \mathcal{F}) = 0$.
\end{thm}
\begin{proof}
As $X$ is obtained from successive blow ups of points from $X_l$, the cohomology groups of $\mathcal{F}$ and 
$\mathcal{F}_l$ are isomorphic \cite[Prop.V.3.4]{Hart:AG}. Hence we assume $l = n$.

\noindent The dualizing sheaf on the Cartier divisor $C$ is 
$\omega_C := \left[ K_X \otimes \mathcal{O}(C) \right]_{|C}$ and Serre duality holds (see~\cite[III.7]{Hart:AG}):
\[ h^1(C, \mathcal{F}_{|C}) = h^0(C, \omega_C \otimes \mathcal{F}_{|C}^{-1}). \]
Using the morphism $\phi: C_1 \sqcup \ldots \sqcup C_r \rightarrow C$ from the disjoint union of the irreducible components
$C_1, \ldots, C_r$ on $C$, the inclusion $\mathcal{O}_C \subset \phi_\ast(\mathcal{O}_{C_1 \sqcup \ldots \sqcup C_r})$ and
the projection formula, we conclude that $H^0(C, \omega_C \otimes \mathcal{F}_{|C}^{-1})$ is a subgroup of
\[ H^0(C_1 \sqcup \ldots \sqcup C_r, \phi^\ast(\omega_C \otimes \mathcal{F}_{|C}^{-1})) = 
   \bigoplus_{i=1}^r H^0(C_i, \left[ K_X \otimes \mathcal{O}_X(C) \otimes \mathcal{F}_{|C}^{-1} \right]_{|C_i}). \]
But since $\left[ K_X \otimes \mathcal{O}_X(C)\right].C_i < \mathcal{F}.C_i$, the degree of the invertible sheaf 
$\left[ K_X \otimes \mathcal{O}_X(C) \otimes \mathcal{F}_{|C}^{-1} \right]_{|C_i}$ is negative on the irreducible curve $C_i$,
hence
\[ h^0(C_i, \left[ K_X \otimes \mathcal{O}_X(C) \otimes \mathcal{F}_{|C}^{-1} \right]_{|C_i}) = 0, \]
and by Serre duality, $H^1(C, \mathcal{F}_{|C}) = 0$.

\noindent The claim follows from considering the long exact cohomology sequence associated to the short exact sequence
\[ 0 \rightarrow \mathcal{F} \otimes \mathcal{O}_{X}(-C) \rightarrow \mathcal{F} \rightarrow \mathcal{F}_{|C} \rightarrow 0. \]
\end{proof}

\noindent Of course, this theorem just shifts the proof of vanishing to another line bundle which hopefully is simpler. For
$\mathbb{F} = \mathbb{P}^2$,
Harbourne \cite{Har85} developped an inductive scheme which guarantees vanishing if $|-K_X|$ contains an irreducible and
reduced section and the coefficients of $\mathcal{F} = \mathcal{L}(d;m_1, \ldots, m_n)$ satisfy some numerical conditions.

\begin{Def}
A surface $X = \mathbb{P}^2(p_1, \ldots, p_n)$ is called \textbf{strongly anticanonical} iff the anticanonical linear system
$|-K_X|$ contains an irreducible and reduced section.

\noindent A line bundle $\mathcal{F} = \mathcal{L}(d;m_1, \ldots, m_n)$ on $X$ is called \textbf{standard} iff
\[ d \geq 0,\ m_i \geq 0,\ m_i - m_{i+1} \geq 0,\ d - m_i - m_j - m_k \geq 0,\ 1 \leq i < j < k \leq n.  \]
$\mathcal{F}$ is called \textbf{excellent} iff $\mathcal{F}$ is standard and $\mathcal{F}.K_X < 0$.
\end{Def}

\begin{rem} \label{ac9-rem}
Let $p_1, \ldots, p_n \in \mathbb{P}^2$ be $n \leq 8$ points in general position on $\mathbb{P}^2$. Then 
$X = \mathbb{P}^2(p_1, \ldots, p_n)$ is strongly anti-canonical: For $8$ points in general position on $\mathbb{P}^2$ there 
always exists a smooth cubic curve passing through the points, hence pulling back to a section of the anticanonical bundle 
$-K_X = \mathcal{L}(3;1^n)$
\end{rem}

\begin{rem}
\noindent The line bundle $\mathcal{L}(d;m_1, \ldots, m_n)$ on $X = \mathbb{P}^2(p_1, \ldots, p_n)$ is excellent if 
$\mathcal{L}(d^\prime;m^\prime_1, \ldots, m^\prime_n)$ is excellent and 
\[ d \geq d^\prime,\ m_i \leq m_i^\prime, m_i \geq m_j\ \mathrm{for\ all\ } 1 \leq i < j \leq n. \]
\end{rem}

\begin{thm}[Harbourne's Criterion \cite{Har85}] \label{Har-crit}
Let $X = \mathbb{P}^2(p_1, \ldots, p_n)$ be strongly anti-canonical and $\mathcal{F}$ an excellent line bundle on $X$. Then:
\[ H^1(X,\mathcal{F}) = 0. \]
\end{thm}
\begin{proof}
The statement is a consequence of Theorem~\ref{curvevan-crit} and another way of writing standard line bundles:

\vspace{0.2cm}

\noindent \textit{Claim.} A line bundle $\mathcal{F}$ on $X$ is standard if, and only if, it can be written as 
\[ \mathcal{F} = a_0 \mathcal{E}_0 + a_1 (\mathcal{E}_0 - \mathcal{E}_1) + 
                         a_2 (2 \mathcal{E}_0 - \mathcal{E}_1 - \mathcal{E}_2) + \sum_{i=3}^n a_i (-K_{X_i}), \]
for some integers $a_0, \ldots, a_n \geq 0$. 

\vspace{0.2cm}

\noindent \textit{Proof of the Claim.} This is \cite[Lem.~1.4]{Har85}. \hfill $\Box$

\vspace{0.2cm}

\noindent Here the anticanonical line bundle $-K_{X_i}$ on 
$X_i = \mathbb{P}^2(p_1, \ldots, p_i)$ is interpreted as a line bundle on $X$ via pullback.

\noindent The proof of the Criterion now proceeds by a double induction on 
\[ \begin{array}{rcl}
   l & := & \min \{ k: \mathcal{F} = \pi_n^\ast \cdots \pi_{k+1}^\ast \mathcal{F}_k,\ \mathcal{F}_k\ \mathrm{line\ bundle\ on\ } 
                                                                                                                             X_k = \mathbb{P}^2(p_1, \ldots, p_k) \} \\
     &  = & \max \{k: a_k \neq 0\} 
   \end{array} \]
and $a_l$: The induction start with $\mathcal{F} = \mathcal{O}_X$ is trivial. For the induction step, we can apply 
Theorem~\ref{curvevan-crit}, because
\begin{itemize}[label=(\roman*), leftmargin=*]
\item[(i)] all the line bundles $\mathcal{E}_0$, $\mathcal{E}_0 - \mathcal{E}_1$,  
$2 \mathcal{E}_0 - \mathcal{E}_1 - \mathcal{E}_2$ and the  $-K_{X_i}$ have an irreducible and reduced section on the
$X_l$ where they are not a pullback: a line, the strict transforms of a line through $p_1$ and of a conic through $p_1, p_2$, the 
images of the $-K_X$-section, 
\item[(ii)] $(K_{X_l} + C_l).C_l < 0$ for $l = 0,1,2$, and $= 0$ for $l \geq 3$,
\item[(iii)] $\mathcal{F}.\mathcal{E}_0 = a_0 > 0 $ for $l = 0$, $\mathcal{F}.(\mathcal{E}_0 - \mathcal{E}_1) = a_0 + a_1 - a_1 = 
                                                                                                  a_0 \geq 0$ for $l = 1$,
                 $\mathcal{F}.(2 \mathcal{E}_0 - \mathcal{E}_1 - \mathcal{E}_2) = 2a_0 + a_1 + 2a_2 > 0$ for $l = 2$, and 
                $\mathcal{F}.(-K_{X_l}) > 0$ for $l \geq 3$.
\end{itemize}

\end{proof}

\noindent Not all the surfaces occuring in the degenerations constructed below are 
strongly anticanonical blow ups of $\mathbb{P}^2$. In these cases, we will try to find curves on which we can iteratively apply 
Theorem~\ref{curvevan-crit}, until we obtain a linear system for which we can show vanishing with Harbourne's Criterion. 

\noindent The next criterion will be useful for checking the surjectivity condition in the Gluing Lemma~\ref{glue-lem}:
\begin{prop} \label{negexc-prop}
Let $X$ be a projective complex surface and $\pi: \widetilde{X} = X(p) \rightarrow X$ the blow up of $X$ in $p$, with
exceptional divisor $E \subset \widetilde{X}$. Let $\mathcal{F}$ be a line bundle on $X$ such that $H^1(X, \mathcal{F}) = 0$. Then:
\[ H^1(\widetilde{X}, \pi^\ast \mathcal{F} \otimes \mathcal{O}(E)) = 0. \]
\end{prop}
\begin{proof}
From the exact sequence
\[ 0 \rightarrow \pi^\ast \mathcal{F} \rightarrow \pi^\ast \mathcal{F} \otimes \mathcal{O}(E) \rightarrow 
                        \pi^\ast \mathcal{F}_{|E}  \otimes \mathcal{O}_{E}(E) = \mathcal{O}_{E}(E) \cong \mathcal{O}_{\mathbb{P}^1}(-1)
      \rightarrow 0 \]
we obtain the exact sequence
\[ H^1(\widetilde{X}, \pi^\ast \mathcal{F}) \rightarrow H^1(\widetilde{X}, \pi^\ast \mathcal{F} \otimes \mathcal{O}(E)) \rightarrow 
   H^1(E, \mathcal{O}_{E}(E)), \]
and the proposition follows from $0 = H^1(X, \mathcal{F}) = H^1(\widetilde{X}, \pi^\ast \mathcal{F})$ and 
$H^1(E, \mathcal{O}_{E}(E)) \cong H^1(\mathbb{P}^1, \mathcal{O}_{\mathbb{P}^1}(-1)) = 0$.
\end{proof}

\subsection{Transforming exceptional configurations} \label{Crem-sec}

\noindent Harbourne's Criterion requires the standardness of line bundles on $\mathbb{P}^2(p_1, \ldots, p_n)$ which depends
on the exceptional configuration. These configurations are not at all unique on a given surface, and often a major step in
applying Harbourne's Criterion is to change them appropriately, by means of \textit{Cremona transformations}. Normally, 
Cremona transformations denote birational self-maps of $\mathbb{P}^2$. But in our context we instead consider the change of
exceptional configurations on the desingularisation of these rational maps. 

\noindent We only use compositions of quadratic Cremona transformations involving $3$ base points. The following lemma
describes the possible configurations of these base points: 
\begin{lem}[\cite{Har85}]  \label{Crem-lem}
Let $X = \mathbb{P}^2(p_1, \ldots, p_n)$ and $p_i, p_j, p_k \in X$ points such that either
\begin{itemize}
\item[(i)] $p_i, p_j, p_k$ are not collinear, or
\item[(ii)] $p_i, p_k \in \mathbb{P}^2$ and $p_j$ infinitely near to $p_i$, but $p_i, p_j, p_k$ are not collinear, or
\item[(iii)] $p_i \in \mathbb{P}^2$, $p_j$ infinitely near to $p_i$ and $p_k$ infinitely near to $p_j$, but not to $p_i$.
\end{itemize} 
Then there exist $p_1^\prime, \ldots, p_n^\prime$ and an isomorphism 
$\mathbb{P}^2(p_1, \ldots, p_n) = \mathbb{P}^2(p_1^\prime, \ldots, p_n^\prime)$ such that 
\[ \begin{array}{c} \mathcal{E}_0 = 2 \mathcal{E}_0^\prime - \mathcal{E}_i^\prime - \mathcal{E}_j^\prime - \mathcal{E}_k^\prime,\ 
   \mathcal{E}_l = \mathcal{E}_l^\prime\ \mathrm{for\ } l \neq i,j,k,\ \\  
   \mathcal{E}_i = \mathcal{E}_0^\prime - \mathcal{E}_j^\prime - \mathcal{E}_k^\prime,\ 
   \mathcal{E}_j = \mathcal{E}_0^\prime - \mathcal{E}_i^\prime - \mathcal{E}_k^\prime,\ 
   \mathcal{E}_k = \mathcal{E}_0^\prime - \mathcal{E}_i^\prime - \mathcal{E}_j^\prime. 
   \end{array} \]
In particular, a line bundle $\mathcal{F} = d \mathcal{E}_0 - m_1 \mathcal{E}_1 - \ldots - \mathcal{E}_n$ can be rewritten as
\begin{multline*}
 \mathcal{F} = (2d - m_i - m_j - m_k) \mathcal{E}_0^\prime - \sum_{l \neq i,j,k}m_l \mathcal{E}_l^\prime - \\
                       - (d-m_j-m_k) \mathcal{E}_i^\prime - (d-m_i-m_k) \mathcal{E}_j^\prime - (d-m_i-m_j) \mathcal{E}_k^\prime. 
\end{multline*}
\end{lem}
\begin{proof}
Note that for all 3 configurations we can renumber the base points $p_1, \ldots, p_k$ such that $i=1, j=2, k=3$. In particular,
the blow ups for the other points $p_4, \ldots, p_n$ are not touched when exchanging $p_1, p_2, p_3$ with 
$p_1^\prime, p_2^\prime, p_3^\prime$. So we can assume w.l.o.g. that $n=3$.

\noindent The proof can be read off the following diagrams. The integers denote self intersections, the arrows infinitely near
points.

\vspace{0.2cm}

(i)

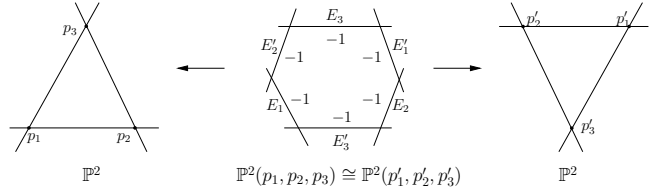
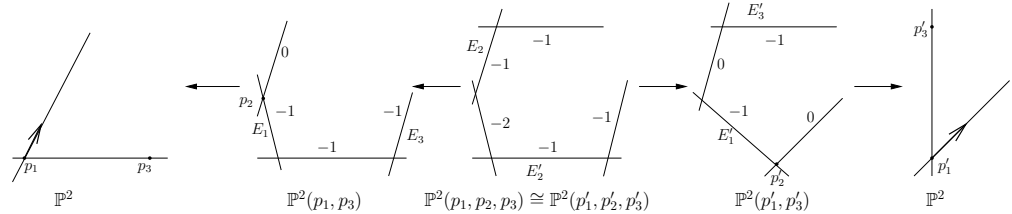
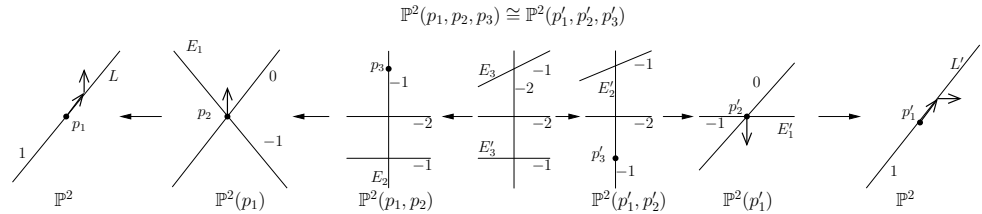
\begin{center}
\scalebox{0.5}{ \input{Cremona1.pstex_t} }
\end{center}


(ii)

\begin{center}
\scalebox{0.5}{ \input{Cremona2.pstex_t} }
\end{center}

(iii)

\begin{center}
\scalebox{0.5}{ \input{Cremona3.pstex_t} }

\end{center}

\vspace{0.2cm}

\noindent We detail the last diagram, the others being even simpler: Let $L$ and $L^\prime$ denote the lines through $p_1$
and $p_1^\prime$. Furthermore, let $\overline{L}$, $\overline{E}_1$, $\overline{E}_2$ and $\overline{E}_3 = E_3$ denote the
strict transforms of $L$, $E_1$, $E_2$ and $E_3$ on $\mathbb{P}^2(p_1,p_2,p_3)$, and $\overline{L}^\prime$, 
$\overline{E}_1^\prime$, $\overline{E}_2^\prime$ and $\overline{E}_3^\prime = E_3^\prime$ the strict transforms of $L^\prime$,
$E_1^\prime$, $E_2^\prime$ and $E_3^\prime$ on $\mathbb{P}^2(p_1^\prime,p_2^\prime,p_3^\prime)$. On
$\mathbb{P}^2(p_1,p_2,p_3) \cong \mathbb{P}^2(p_1^\prime,p_2^\prime,p_3^\prime)$ we can identify
\[ \overline{L} = E_3^\prime,\  \overline{E}_1 = \overline{E}_1^\prime,\  \overline{E}_2 = \overline{E}_2^\prime,\  
   \overline{E}_3 = \overline{L}^\prime. \] 
The claim follows from the equalities 
\[ \mathcal{E}_0 = \overline{L} + \overline{E}_1 + 2\overline{E}_2 + 2\overline{E}_3,\  
   \mathcal{E}_1 = \overline{E}_1 + \overline{E}_2 + \overline{E}_3,\  \mathcal{E}_2 = \overline{E}_2 + \overline{E}_3, \ 
   \mathcal{E}_3 = E_3 \]
on $\mathbb{P}^2(p_1,p_2,p_3)$ and
\[ \overline{L}^\prime = \mathcal{E}_0^\prime - \mathcal{E}_1^\prime - \mathcal{E}_2^\prime,\  
   \overline{E}_1^\prime =  \mathcal{E}_1^\prime - \mathcal{E}_2^\prime, \ 
   \overline{E}_2^\prime =  \mathcal{E}_2^\prime - \mathcal{E}_3^\prime,\  \overline{E}_3^\prime =  \mathcal{E}_3^\prime \]
on $\mathbb{P}^2(p_1^\prime,p_2^\prime,p_3^\prime)$.
\end{proof}

\begin{rem}
Even if the blow ups of $p_4, \ldots, p_n$ are not touched by the Cremona transformations, these points might become 
infinitely near to $p_1^\prime,p_2^\prime,p_3^\prime$. See the Constructions below.
\end{rem}

\noindent We must also know the position of the $p_1^\prime, \ldots, p_n^\prime$ relative to each other. The following
statement is the most general one in this direction, and is almost always implicitely applied:
\begin{prop}
Assume that $p_1, \ldots, p_n$ on $\mathbb{P}^2$ do not contain infinitely near points and are in general position.

\noindent If $\mathbb{P}^2(p_1, \ldots, p_n) \cong \mathbb{P}^2(p_1^\prime, \ldots, p_n^\prime)$ by means of a quadratic
Cremona
transformation as described in Lemma~\ref{Crem-lem}, case (i), then the $p_1^\prime, \ldots, p_n^\prime$ do also not contain 
infinitely near points and are in general position.
\end{prop}
\begin{proof}
A quadratic Cremona transformation as described in Lemma~\ref{Crem-lem}, case (i), induces a birational map from 
$\mathbb{P}^2$
onto itself, which is an isomorphism outside the lines connecting the $3$ base points. Since $p_4, \ldots, p_n$ are in general
position, they do not lie on these lines. The claim follows because $3$ points on $\mathbb{P}^2$ can be freely moved around
by the action of $\mathrm{PGL}(3)$, hence are always in general position. 
\end{proof}

\noindent In some situations some of the points $p_4, \ldots, p_n$ will not be in general position relative to $p_1, p_2, p_3$.
We collect the configurations relevant in the constructions below: 

\vspace{0.2cm}

\noindent
\textbf{Cremona Transformation I.} Let $p_1, \ldots, p_5$ be points on $\mathbb{P}^2$ such that
\begin{itemize}
\item $p_1, p_2, p_4$ are not collinear, 
\item $p_3$ is infinitely near to $p_2$, directed to $p_1$ and
\item $p_5$ is infinitely near to $p_4$, directed to $p_1$. 
\end{itemize}
After a Cremona transformation with base points $p_2, p_3, p_4$, the new exceptional configuration is again of the type
above, with base points $p_2^\prime, p_3^\prime, p_4^\prime$, and the infinitely near points $p_3^\prime, p_1^\prime$ 
are directed to $p_5^\prime$. This can be read off the following diagram:

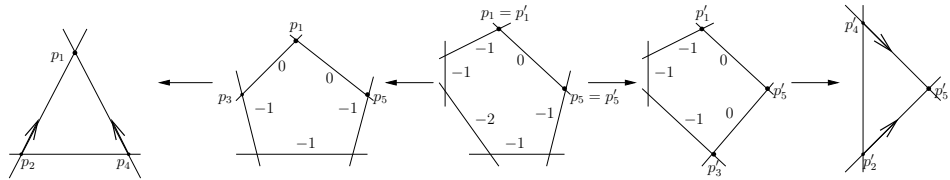
\vspace{0.2cm}

\begin{center}
\scalebox{0.5}{ \input{Cremona5.pstex_t} }
\end{center}

\vspace{0.2cm}

\noindent
\textbf{Cremona Transformation II.} Let $p_1, \ldots, p_5$ be points on $\mathbb{P}^2$ such that
\begin{itemize}
\item $p_1, p_2, p_4$ are not collinear, 
\item $p_3$ is infinitely near to $p_2$, but not directed to $p_1$ or $p_4$ and
\item $p_5$ is infinitely near to $p_4$, but not directed to $p_1$ or $p_2$. 
\end{itemize}
After two Cremona transformation with base points $p_1, p_2, p_4$ and $p_1^\prime, p_3^\prime, p_5^\prime$, the new exceptional configuration is again of the type
above, with base points $p_1^{\prime\prime}, p_3^{\prime\prime}, p_4^{\prime\prime}$, and the infinitely near points 
$p_2^{\prime\prime}, p_4^{\prime\prime}$ are not directed to $p_1^{\prime\prime}, p_4^{\prime\prime}$ resp.
$p_1^{\prime\prime}, p_3^{\prime\prime}$. This can be read off the following diagram:

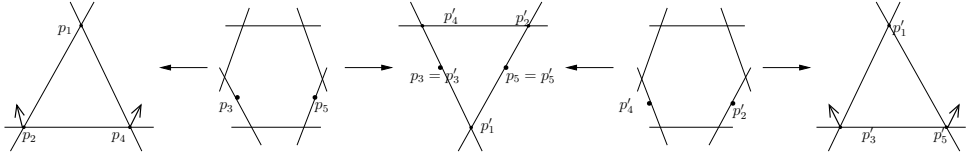
\vspace{0.2cm}

\begin{center}
\scalebox{0.5}{ \input{Cremona4.pstex_t} }
\end{center}

\vspace{0.2cm}

\noindent
\textbf{Cremona Transformation III.} Let $p_1, \ldots, p_5$ be points on $\mathbb{P}^2$ such that
\begin{itemize}
\item No three of the points $p_1, p_2, p_3, p_4$ are collinear, 
\item $p_5$ is infinitely near to $p_4$, directed to $p_1$. 
\end{itemize}
After a Cremona transformation with base points $p_1, p_2, p_3$ the point $p_5^\prime$ infinitely near to $p_4^\prime$ is
directed to $p_1^\prime$. This can be read off the following diagram:

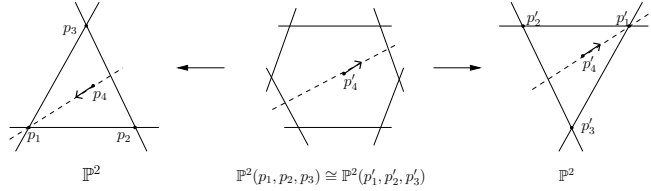
\vspace{0.2cm}

\begin{center}
\scalebox{0.5}{ \input{Cremona6.pstex_t} }
\end{center}

\vspace{0.2cm}

\noindent Finally, $p_1, \ldots, p_n$ might not be in general position because they lie on a special curve $C$. This curve can
be interpreted as the section of a line bundle $\mathcal{L}(d;m_1, \ldots, m_n)$, which is transformed by a Cremona transformation --- or several of them --- to $\mathcal{L}(d^\prime;m_1^\prime, \ldots, m_n^\prime)$ resp. $C^\prime$. Thus we
can translate the special position of the $p_1, \ldots, p_n$ into a special position of the $p_1^\prime, \ldots, p_n^\prime$. 

\subsection{Throwing curves}

\noindent Consider the setting of Prop.~\ref{degen-prop}: Let $f: \mathcal{X} \rightarrow \Delta$ be a family of complex varieties
over the unit disc $\Delta$, and let $\mathcal{L}$ be a line bundle on $\mathcal{X}$. Then the cohomology group 
$H^1(X_t, L_t)$ vanishes if $H^1(X_0, L_0) = 0$. To show this we want to apply the Gluing Lemma~\ref{glue-lem}. Its
application  fails if $H^1(V,\mathcal{L}_{|V}) \neq 0$ on an irreducible component $V \subset X_0$. For $\mathbb{P}^2$ or Hirzebruch surfaces $\mathbb{F}_k$ blown up in several points
this is the case if there exists a $(-1)$-curve $E$ on $V$ such that 
$\mathcal{L}_{|V}.E \leq -2$ and $\mathcal{L}_{|V}$ has a global section:
\begin{lem}[\cite{Har85, CilMir98}] \label{spec-lem}
Let $\mathbb{F}$ be $\mathbb{P}^2$ or a Hirzebruch surface $\mathbb{F}_k$, and $V = \mathbb{F}(p_1, \ldots, p_n)$ be a blow
up of $n$ points in $\mathbb{F}$. Let $L$ be an effective line bundle on $V$.
Assume that $E$ is a $(-1)$-curve on $V$ such that $L.E \leq -2$. Then:
\[ H^1(V, L) \neq 0. \]
\end{lem}
\begin{proof}
Set $L.E = -k$, $k \geq 2$ an integer. From the short exact sequence 
\[ 0 \rightarrow L \otimes \mathcal{O}(-E) \rightarrow L \rightarrow L_{|E} \rightarrow 0 \]
we obtain  the following part of the long exact cohomology sequence:
\[ H^1(V, L) \rightarrow H^1(E, L_{|E}) \rightarrow H^2(V, L \otimes \mathcal{O}(-E)). \]
Next, we calculate
\[ H^1(E, L_{|E}) \cong H^1(\mathbb{P}^1, \mathcal{O}(-k)) \cong H^0(\mathbb{P}^1, \mathcal{O}(-2+k))^\vee \neq 0. \]
Since $L$ has a global section, $L.E < 0$ implies that $E$ is a fixed divisor in the associated linear system. Consequently,
$L-E$ is also effective , and $(L-E).\mathcal{F} \geq 0$ for every nef divisor $\mathcal{F}$. For 
$\mathbb{F} = \mathbb{P}^2$ let $\mathcal{F}$ be the pull
back of a line on $\mathbb{P}^2$, for $\mathbb{F} = \mathbb{F}_k$ let $\mathcal{F}$ be the pull back of a fiber in the 
$\mathbb{P}^1$-bundle $\mathbb{F}_k$. In both cases $K_V.\mathcal{F} < 0$, and this implies
\[ (K_V \otimes L^\vee \otimes \mathcal{O}(E)).\mathcal{F} < 0, \]
hence $K_V \otimes L^\vee \otimes \mathcal{O}(E))$ cannot be effective. Consequently,
\[ H^2(V, L \otimes \mathcal{O}(-E)) \cong H^0(V, K_V \otimes L^\vee \otimes \mathcal{O}(E))^\vee = 0. \] 
The lemma follows.
\end{proof}

\noindent The new idea in \cite{CM08} is to change the degeneration $f: \mathcal{X} \rightarrow \Delta$ and the line bundle
$\mathcal{L}$, whenever such a "bad" $(-1)$-curve as in the lemma occurs in one of the components of $X_0$, by flopping it.
Such a flop certainly exists if the $(-1)$-curve $E$ has normal bundle 
$N_{E/\mathcal{X}} \cong \mathcal{O}_{\mathbb{P}^1}(-1) \oplus \mathcal{O}_{\mathbb{P}^1}(-1)$: it is the Atiyah flop 
(see \cite[Ex.3-4-3]{Mat02}). This is not always the case, but the following lemma shows that the normal bundle is at least
always negative. Hence we can improve the normal bundle by blowing up $\mathcal{X}$ several times along $E$ resp. its
strict transforms, until it is possible to flop $E$. The flop contracts $E$ on the component $V$, but other curves pop up on
different components. Therefore this operation is called a "\textbf{throw}". 
\begin{lem}[Three-point formula]  \label{three-lem}
Let $f: \mathcal{X} \rightarrow \Delta$ be a projective fibration from a smooth $3$-fold $\mathcal{X}$ such that 
$X_0 = \bigcup V_i$ is a union of smooth projective surfaces. 

\noindent Suppose that $C \subset V_i$ is a $(-1)$-curve on $V_i$ not contained in any other component $V_j$, and set
$s := \sum_{i \neq j} C.V_j$. Then
\[ N_{C/\mathcal{X}} \cong N_{C/V_i} \oplus N_{V_i/\mathcal{X}|C} \cong 
   \mathcal{O}_{\mathbb{P}^1}(-1) \oplus \mathcal{O}_{\mathbb{P}^1}(-s). \] 
\end{lem}
\begin{proof}
Since $f$ is a fibration, 
$\mathcal{O}_{\mathcal{X}} = \mathcal{O}_{\mathcal{X}}(X_0) = \bigotimes \mathcal{O}_{\mathcal{X}}(V_j)$. Hence
$\mathcal{O}_C = \bigotimes \mathcal{O}_C(V_j)$, and 
\[ \mathcal{O}_C(V_i) = \bigotimes_{i \neq j} \mathcal{O}_C(-V_j) \cong \mathcal{O}_{\mathbb{P}^1}(-s). \]
The claim follows because there is a natural bundle surjection 
\[ N_{C/\mathcal{X}} \rightarrow N_{V_i/\mathcal{X}|C} \cong \mathcal{O}_C(V_i) \]
whose kernel is $T_{V_i|C}/T_C = N_{C/V_i} \cong \mathcal{O}_{\mathbb{P}^1}(-1)$.
\end{proof}

\noindent After the flop the strict transforms of other components besides $V$ can be singular. To analyse line bundles on
these singular surfaces, Ciliberto and Miranda use the desingularization given by the blow up part of the Atiyah flop. To avoid 
this additional technical difficulty we present the throwing procedure as a sequence of blow ups only.
\begin{const}[Throwing $(-1)$-curves]  \label{throw-const}
Let $f: \mathcal{X} \rightarrow \Delta$ be a projective fibration from a smooth complex $3$-fold $\mathcal{X}$ to the unit disc
$\Delta$ such that the central fiber $X_0$ is a union $\bigcup V_i$ of smooth projective surfaces. Let $\mathcal{L}$ be a line
bundle on $\mathcal{X}$.

\noindent Assume that $C$ is a $(-1)$-curve on a component $V_i$ such that for $j \neq i$ the intersection 
$C \cap V_j$ consists of $s_j$ points with multiplicity $1$, and $C \cap V_j \cap V_{j^\prime} = \emptyset$ for 
$j, j^\prime \neq i$. Set $n := \sum_{j \neq i} s_j$ and $l := - \mathcal{L}_{|V_i}.C$.

\noindent Then we construct a sequence of blow ups
\[ \widetilde{\mathcal{X}} = \mathcal{X}_n \stackrel{\pi_n}{\rightarrow} \mathcal{X}_{n-1} 
   \stackrel{\pi_{n-1}}{\rightarrow} \cdots 
                           \stackrel{\pi_2}{\rightarrow} \mathcal{X}_1 \stackrel{\pi_1}{\rightarrow} \mathcal{X}_0 = \mathcal{X} \]
such that the center of $\pi_1$ is $C \subset \mathcal{X}$ and for $k=2, \ldots, n$, the center $C_k \cong C$ of $\pi_k$ is the
intersection of the strict
transform $V_i^{(k-1)} \cong V_i$ of $V_i$ on $\mathcal{X}_{k-1}$ with the exceptional divisor $T_{k-1}$ of 
$\pi_{k-1}$.

\noindent Setting $\phi_k := \pi_k \circ \cdots \circ \pi_n$ we denote by $\widetilde{T}_k$ the strict $(\phi_{k+1})$-transform of
$T_k \subset \mathcal{X}_k$, for $k=1, \ldots, n-1$. Note that $\widetilde{T}_k \cong T_k$, by construction.

\noindent Finally we define a new line bundle on $\widetilde{\mathcal{X}}$:
\[ \widetilde{\mathcal{L}} := 
   \phi_n^\ast \mathcal{L} \otimes \mathcal{O}_{\widetilde{\mathcal{X}}}(\sum_{l=1}^n a_l \widetilde{T}_l). \]
This construction is called an \textbf{$n$-throw}. It has the following properties:
\begin{enumerate}
\item $T_k \cong \mathbb{F}_{n-k}$, the $(n-k)$th Hirzebruch surface.
\item $\widetilde{V}_j$ is the blow up of sequences of $n$ infinitely near points $p_1, \ldots, p_n$ over each point 
$p = p_1 \in C \cap V_j$, where $p_i$ is infinitely near to $p_{i-1}$ but not to $p_{i-2}$, $i \geq 3$. In particular the choice of 
$C$ fixes these points.
\item $\widetilde{T}_k \cap \widetilde{T}_{k^\prime} = \emptyset$ if $k < k^\prime - 1$, whereas 
$\widetilde{T}_k \cap \widetilde{T}_{k+1}$ is the curve at infinity on 
$\widetilde{T}_k \cong \mathbb{F}_{n-k}$, and is a section of $\widetilde{T}_{k+1}$ not intersecting its curve at infinity 
$E_{n-k-1}$ and linearly equivalent to $E_{n-k-1} + (n-k-1)F_{n-k-1}$ (see the notation in the introduction). 
\item $\widetilde{T}_k \cap \widetilde{V}_i = \emptyset$ if $k < n$, whereas $T_n \cap \widetilde{V}_i$ is 
$C$ on $\widetilde{V}_i \cong V_i$, and a horizontal $\mathbb{P}^1$-fiber on $T_n \cong \mathbb{P}^1 \times \mathbb{P}^1$.
\item For $j \neq i$ the intersection $\widetilde{T}_k \cap \widetilde{V}_j$ consists of the exceptional divisors of the $k$th blow
ups over points in $C \cap V_j$ on $\widetilde{V}_j$, which are $\mathbb{P}^1$-fibers of 
$\widetilde{T}_k \cong \mathbb{F}_{n-k}$ on $\widetilde{T}_k$.
\item The strict transform of an irreducible intersection curve in $V_i \cap V_j$ is linearly equivalent to the pullback of the
intersection curve minus the exceptional divisors over points in $V_j \cap C$ on this curve.
\item $\widetilde{\mathcal{L}}_{|\widetilde{V}_i} = \phi_n^\ast \mathcal{L}_{|\widetilde{V}_i} \otimes 
           \mathcal{O}_{\widetilde{V}_i}(a_n C_n)$, and this line bundle is trivial on $C_n$ iff $a_n = -k$.
\item If $\mathcal{E}_1, \ldots, \mathcal{E}_n$ is the configuration of exceptional divisors over an intersection point in 
$V_j \cap C$ for $j \neq i$, the divisor $\mathcal{E}_m$ occurs in $\widetilde{\mathcal{L}}_{|\widetilde{V}_j}$ with multiplicity
$- a_m + a_{m-1}$ (in the notation of the introduction).
\item $\widetilde{\mathcal{L}}_{|T_n} \cong \mathcal{O}_{\mathbb{P}^1 \times \mathbb{P}^1}(-k-a_n, a_{n-1}-a_n)$.
\item For $k < n$, 
\[ \widetilde{\mathcal{L}}_{|\widetilde{T}_k} \cong 
         \mathcal{O}_{\mathbb{F}_{n-k}}((a_{k+1}-2a_k+a_{k-1})E_{n-k} + (-l - a_k(n-k+1) + a_{k-1}(n-k))F_{n-k}). \]    
\end{enumerate} 
\end{const}
\begin{proof}
The Three-point Formula~\ref{three-lem} yields
\[ N_{C/\mathcal{X}} = N_{C/V_i} \oplus N_{V_i/\mathcal{X}|C} \cong 
   \mathcal{O}_{\mathbb{P}^1}(-1) \oplus \mathcal{O}_{\mathbb{P}^1}(-n). \]
If we projectivize $N_{C/\mathcal{X}}$ the intersection curve of this $\mathbb{P}^1$-bundle with $V_i^{(1)}$ comes from the 
summand $\mathcal{O}_{\mathbb{P}^1}(-1)$ hence is the curve at infinity on $T_1 \cong \mathbb{F}_{n-1}$.
\[ N_{C_k/\mathcal{X}_{k-1}} = N_{C_k/V_i^{(k-1)}} \oplus N_{C_k/T_{k-1}}, \]
because $V_i^{(k-1)}$ and $T_{k-1}$ intersect transversally in $C_k \cong \mathbb{P}^1$,
\[ N_{C_k/T_{k-1}} \cong \mathcal{O}_{\mathbb{P}^1}(-n+k-1), \]
because by induction $C_k$ is the curve at infinity of $T_{k-1} \cong \mathbb{F}_{n-k+1}$. Consequently,
\[ N_{C_k/\mathcal{X}_{k-1}} \cong  \mathcal{O}_{\mathbb{P}^1}(-1) \oplus \mathcal{O}_{\mathbb{P}^1}(-n+k-1), \]
which yields (1).

\noindent (2), (3), (4), (5) follow from construction, (6) is true because $C$ intersects the components $V_j$, $j \neq i$, 
transversally. 

\noindent (7) is obvious from the definition of $\widetilde{\mathcal{L}}$, and the intersection configurations described in (4).

\noindent (8) follows from (6) and the fact that $\mathcal{E}_m$ contains every exceptional divisor $E_{m^\prime}$ (resp. its
strict transform) exactly once, if $m^\prime \geq m$. 

\noindent (9) is the result of the following calculation:
\begin{eqnarray*}
\widetilde{\mathcal{L}}_{|T_n} & = &\phi_n^\ast \mathcal{L}_{|T_n} \otimes \mathcal{O}_{T_n}(a_n T_n) \otimes  
                                                                                                                     \mathcal{O}_{T_n}(a_{n-1} \widetilde{T}_{n-1}) \\
  & \cong & \mathcal{O}_{\mathbb{P}^1 \times \mathbb{P}^1}(\mathcal{L}_{|V_i}.C, 0) \otimes 
                  \mathcal{O}_{\mathbb{P}^1 \times \mathbb{P}^1}(-a_n, -a_n) \otimes 
                  \mathcal{O}_{\mathbb{P}^1 \times \mathbb{P}^1}(0, a_{n-1})
\end{eqnarray*} 
because
\begin{eqnarray*}
 \mathcal{O}_{T_n}(T_n) & = & N_{T_n/\widetilde{\mathcal{X}}} = \mathcal{O}_{\mathbb{P}(N_{C_n/\mathcal{X}_{n-1}})} = 
   \mathcal{O}_{\mathbb{P}(\mathcal{O}_{\mathbb{P}^1}(1) \oplus \mathcal{O}_{\mathbb{P}^1}(1))}(-1) = \\
  & = & \mathcal{O}_{\mathbb{P}(\mathcal{O}_{\mathbb{P}^1} \oplus \mathcal{O}_{\mathbb{P}^1})}(-1) \otimes 
   p_1^\ast \mathcal{O}_{\mathbb{P}^1}(-1) = \mathcal{O}_{\mathbb{P}^1 \times \mathbb{P}^1}(-1, -1). 
\end{eqnarray*}

\noindent (10) is a consequence of
\[ \widetilde{\mathcal{L}}_{|\widetilde{T}_k} = \phi_n^\ast \mathcal{L}_{|\widetilde{T}_k} \otimes 
   \mathcal{O}_{\widetilde{T}_k}(a_{k+1} \widetilde{T}_{k+1}) \otimes  
   \mathcal{O}_{\widetilde{T}_k}(a_k \widetilde{T}_k) \otimes
   \mathcal{O}_{\widetilde{T}_k}(a_{k-1} \widetilde{T}_{k-1}), \]
\begin{eqnarray*} 
   \lefteqn{\phi_n^\ast \mathcal{L}_{|\widetilde{T}_k} \otimes \mathcal{O}_{\widetilde{T}_k}(a_{k+1} \widetilde{T}_{k+1}) 
                \otimes \mathcal{O}_{\widetilde{T}_k}(a_{k-1} \widetilde{T}_{k-1}) \cong} \\ 
    & & \mathcal{O}_{\mathbb{F}_{n-k}}(-l F_{n-k}) \otimes \mathcal{O}_{\mathbb{F}_{n-k}}(a_{k+1} E_{n-k}) \otimes 
           \mathcal{O}_{\mathbb{F}_{n-k}}(a_{k-1}(E_{n-k} + (n-k) F_{n-k})), 
\end{eqnarray*}
\[ \widetilde{T}_k = \phi_{k+1}^\ast T_k - \widetilde{T}_{k+1} - \ldots - T_n \]
and
\[ \mathcal{O}_{T_k}(T_k) \cong \mathcal{O}_{\mathbb{F}_{n-k}}(-E_{n-k} - (n-k+1) F_{n-k}). \] 
\end{proof}

\begin{rem}
Note that in our description of a throw, we do not contract the "bad" $(-1)$-curve $C$ and push down the line bundle 
$\mathcal{L}$, but we only change the line bundle until it is trivial on $C$.
\end{rem}

\begin{rem}
In the examples of throws below we choose $a_1, \ldots, a_{n-1}$ such that the restrictions of the line bundle $\widetilde{L}$
to the exceptional divisors $\widetilde{T}_k$ become minimal.
\end{rem}

\begin{rem}
Ciliberto and Miranda \cite{CM08} only need $1$- and $2$-throws. But we will see in Section~\ref{Alg-sec} that
more blow ups can be necessary. 
\end{rem}

\subsection{Bounds from linear inequalities}

\noindent Applying the Gluing Lemma~\ref{glue-lem} requires the vanishing of $H^1(V_i, \mathcal{L}_{|V_i})$ on components 
$V_i$ of $X_0$. In the Ciliberto-Miranda degenerations constructed in section~\ref{CM-deg-sec} below, 
$V_i$ is always a blow up of $\mathbb{P}^2$ or a Hirzebruch surface $\mathbb{F}_k$ in points $p_1, \ldots, p_n$, and
$\mathcal{L}_{|V_i} \cong \mathcal{L}(d; m_1, \ldots, m_n)$  resp.\ 
$\mathcal{L}_{|V_i} \cong \mathcal{L}(d_1,d_2; m_1, \ldots, m_n)$ is linearly depending on parameters $d$ resp.\ $d_1,d_2$,
$m_1, \ldots, m_n$. After possibly
performing some Cremona transformations, we would like to use the
criteria in section~\ref{non-spec-sec} to deduce the vanishing of the first cohomology group.

\noindent It turns out in section~\ref{CM-deg-sec} that this is possible on the occuring varieties 
whenever the integers $d$ resp. $d_1, d_2$, 
$m_1, \ldots, m_n$ satisfy a set of linear inequalities. Together with the Gluing Lemma, this observation can be used to find
$d,m$ arbitrarily big such that $\mathcal{L}(d;m^n)$ is non-special on $\mathbb{P}^2(p_1, \ldots, p_n)$, with $p_1, \ldots, p_n$
in general position:
\begin{thm} \label{Ineq-thm}
Let $f: \mathcal{X} \rightarrow \Delta$ be a projective fibration from a smooth complex $3$-fold such that for $t \neq 0$, 
$X_t \cong \mathbb{P}^2(p_1, \ldots, p_n)$, $n > 9$, with $p_1, \ldots, p_n$ in general position, and 
$X_0 \cong \bigcup V_i$, all the $V_i \cong \mathbb{P}^2(q_1^{(i)}, \ldots, q_{n_i}^{(i)})$ or
$\cong \mathbb{F}_k(q_1^{(i)}, \ldots, q_{n_i}^{(i)})$. Furthermore, denote by $C_i$ the intersection curve 
$\bigcup_{j < i} V_j \cap V_i$. 

\noindent Suppose that there exists $k \in \mathbb{N}$ such that for every $d, m, a \in k \cdot \mathbb{N}$,
we can construct a line bundle 
$\mathcal{L} = \mathcal{L}(d,m,a)$ satisfying the following conditions:
\begin{itemize}
\item $\mathcal{L}_{|X_t} \cong \mathcal{L}(d;m^n)$ for $t \neq 0$, 
\item $\mathcal{L}_{|V_i} \cong \mathcal{L}(d^{(i)}; m_1^{(i)}, \ldots, m_{n_i}^{(i)})$ resp. 
$\mathcal{L}(d^{(i)}_1, d^{(i)}_2; m_1^{(i)}, \ldots, m_{n_i}^{(i)})$, the $d^{(i)}$ resp. $d^{(i)}_1, d^{(i)}_2$, 
$m_1^{(i)}, \ldots, m_{n_i}^{(i)}$ depending linearly on $d,m,a$, and
\item if the $d^{(i)}$ resp. $d^{(i)}_1, d^{(i)}_2$, $m_1^{(i)}, \ldots, m_{n_i}^{(i)}$ satisfy a finite set of weak linear inequalities 
then 
\[ H^1(V_i, \mathcal{L}_{|V_i}) = H^1(V_i, \mathcal{L}_{|V_i} \otimes \mathcal{O}(-C_i)) = 0. \] 
\end{itemize}
Substituting $d, m, a$ in the $d^{(i)}$ resp. $d^{(i)}_1, d^{(i)}_2$, $m_1^{(i)}, \ldots, m_{n_i}^{(i)}$, we can consider the closed
convex polyhedron 
$P \subset \mathbb{R}^3$ described by the resulting weak linear inequalities in $d, m, a$, and its projection 
$P^\prime \subset \mathbb{R}^2$ onto the $d-m$-coordinates. Set
\[ \mu := \inf \left\{ \frac{d}{m}: (d,m) \in P^\prime \right\}. \]
If $P^\prime$ is unbounded, both in $d$ and in $m$, then there exist $\epsilon > 0$, $M > 0$, such that for all integers 
$d, m > M$ with $d,m \in k \cdot \mathbb{N}$ and $0 \leq \frac{d}{m} - \mu \leq \epsilon$, the line bundle 
$\mathcal{L}(d; m^n)$ is a non-special line bundle on the blow
up of $\mathbb{P}^2$ in $n$ points in general position.

\noindent If $\mu > \sqrt{n}$ then the line bundle $\mathcal{L}(d; m^n)$ is furthermore effective, for $M \gg 0$. 
\end{thm}   
\begin{proof}
Note that for some positive integers
$c^{(i)}$ resp. $c^{(i)}_1, c^{(i)}_2$,  $n_1^{(i)}, \ldots, n_{n_i}^{(i)}$ not depending on $d, m, a$ the intersection curve $C_i$ is a
section of the line bundle $\mathcal{L}(c^{(i)}; n_1^{(i)}, \ldots, n_{n_i}^{(i)})$ on $\mathbb{P}^2(q_1^{(i)}, \ldots, q_{n_i}^{(i)})$. resp. $\mathcal{L}(c^{(i)}_1,c^{(i)}_2; n_1^{(i)}, \ldots, n_{n_i}^{(i)})$ on $\mathbb{F}_k(q_1^{(i)}, \ldots, q_{n_i}^{(i)})$. Consequently the vanishing of 
$H^1(V_i, \mathcal{L}_{|V_i} \otimes \mathcal{O}(-C_i))$ can also be deduced from a set of weak linear inequalities depending
on $d, m, a$.

\noindent The unboundedness implies that there is a line with slope $\mu$ bounding the convex polytope $P^\prime$ from
below in the
region $\{ m \geq M \}$, and an $\epsilon > 0$ such that all pairs $(d,m)$ with $d,m > M$, 
$0 \leq \frac{d}{m} - \mu \leq \epsilon$ lie in $P^\prime$. For such pairs $(d,m) \in k \cdot \mathbb{N}^2$, the assumptions tell
us
\[ H^1(V_i, \mathcal{L}_{|V_i}) = H^1(V_i, \mathcal{L}_{|V_i} \otimes \mathcal{O}(-C_i)) = 0, \]
hence both conditions of the Gluing Lemma~\ref{glue-lem} are satisfied (use Remark~\ref{diffmap-rem} for the surjectivity of the
difference map). Consequently $H^1(X_t, \mathcal{L}_t) = 0$ for $t \in \Delta$ general, by Prop.~\ref{degen-prop}.

\noindent Since $\mathcal{E}_0.(K_{X_t} \otimes L_t^\vee) = -3 - d < 0$, the divisor $K_{X_t} \otimes L_t^\vee$ cannot be 
effective, and $h^2(X_t, L_t) = h^0(X_t, K_{X_t} \otimes L_t^\vee) = 0$. Furthermore, 
$p_a = h^0(X_t, \mathcal{O}_{X_t}) = h^0(\mathbb{P}^2, \mathcal{O}_{\mathbb{P}^2}) = 1$ because $X_t$ is a blow up of
$\mathbb{P}^2$ (\cite[Prop.V.3.4]{Hart:AG}). Consequently, Riemann-Roch implies
\[ h^0(X_t, L_t) = \frac{1}{2} L_t.(L_t - K_{X_t}) = \frac{1}{2} ( d(d+3) - n \cdot m(m+1) ) + 1 > 0 \]
for $d > \sqrt{n} m$ and $d, m \gg 0$.
\end{proof}

\begin{rem}
The $k$ must be introduced because the $d^{(i)}$ resp. $d^{(i)}_1, d^{(i)}_2$, $m_1^{(i)}, \ldots, m_{n_i}^{(i)}$ can linearly 
depend on $d, m, a$ with rational
coefficients. Then $k$ is a common denominator for all occuring fractions.  
\end{rem}

\begin{rem}
We can use $H^1(V_1, \mathcal{L}_{|V_1} \otimes \mathcal{O}(-C_2)) = 0$ instead of 
$H^1(V_2, \mathcal{L}_{|V_2} \otimes \mathcal{O}(-C_2)) = 0$ to show the surjectivity of the first difference map.
\end{rem}

\noindent We can use the information obtained from the last theorem to deduce lower bounds for Seshadri constants:
\begin{prop}
Assume that for all $\epsilon > 0$, $M \gg 0$ there exist $d, m > M$ with $0 \leq \frac{d}{m} - \mu \leq \epsilon$ such that
$\mathcal{L}(d; m^n)$ is non-empty and non-special. Then the multi-point Seshadri constant for $n$ points in general position
on $\mathbb{P}^2$ is bounded by
\[ \epsilon(\mathbb{P}^2, \mathcal{O}_{\mathbb{P}^2}(1); p_1, \ldots, p_n) \geq \frac{1}{\mu}. \]
\end{prop}
\begin{proof}
From the assumptions we construct a sequence $(d_k,m_k)$ of
monotonely increasing integers with $d_k, m_k \rightarrow \infty$ for $k \rightarrow \infty$ such that 
$\frac{d_k}{m_k} \rightarrow \mu$ and $\mathcal{L}(d_k; m_k^n)$ is non-empty and non-special. Since also 
$\frac{d_k}{m_k-1} \rightarrow \mu$, we can apply Theorem~\ref{Approx-thm} from~\cite{Eck08}. 
\end{proof}

\section{Degenerations of $\mathbb{CP}^2$ blown up in $10$ points}  \label{CM-deg-sec}

\noindent As Ciliberto and Miranda in \cite{CM08} we exemplify their method on $\mathbb{P}^2$ blown up in $10$ points, thus
being 
able to compare the arguments. But of course, it can also be applied to $\mathbb{P}^2$ blown up in more than $10$ points.

\subsection{The First Degeneration} 

\noindent The starting point is a degeneration constructed by Ciliberto and Miranda in \cite{CilMir98}: Blow up 
$\mathbb{P}^2 \times \Delta$ in a point $p \in \mathbb{P}^2 \times \{0\}$, and obtain the projective fibration 
$\pi: \mathcal{X} \rightarrow \Delta$. Its central fibre decomposes into the exceptional divisor $P_p \cong \mathbb{P}^2$ and
$F_p \cong \mathbb{P}^2$, the strict transform of $\mathbb{P}^2 \times \{0\}$. 

\noindent Choose $10$ sections $p_1, \ldots, p_{10}: \Delta \rightarrow \mathcal{X}$ such that $p_1(0), \ldots, p_4(0) \in P_p$ resp. $p_5(0), \ldots, p_{10}(0) \in F_p$ are $4$ resp. $6$ points in general position. In particular, $p, p_5(0), \ldots, p_{10}(0)$
are $7$ points in general position on $\mathbb{P}^2$. By possibly shrinking $\Delta$ we can assume w.l.o.g. that for all 
$t \in \Delta$ the points $p_1(t), \ldots, p_{10}(t)$ are in general position on $\mathbb{P}^2$.

\noindent Blowing up the images $p_1(\Delta), \ldots, p_{10}(\Delta)$ of the sections yields a projective fibration 
$\pi_1: \mathcal{X}_1 \rightarrow \Delta$ such that 
\begin{itemize}
\item for all $t \in \Delta$, the fibre $X_{1,t} \cong \mathbb{P}^2(p_1, \ldots p_{10})$ with $p_1, \ldots, p_{10}$ in general position,
and
\item $X_0 = P_1 \cup F_1$ with $P_1 \cong \mathbb{P}^2(p_1, \ldots p_4)$ and 
$F_1 \cong \mathbb{P}^2(p, p_5, \ldots p_{10})$, all these points in general position.
\end{itemize}

\noindent Denote the exceptional divisor over $p_i(\Delta)$ by $E_i$.

\noindent $C_1 = P_1 \cap F_1$ is the pullback of a line on $P_1$, that is a section of $\mathcal{L}(1; 0^4)$, and the 
exceptional divisor over $p$ on $F_1$, that is a section of $\mathcal{L}(0; -1, 0^6)$.

\noindent From the construction of $\mathcal{X}_1$ we obtain a projection 
$f: \mathcal{X}_1 \rightarrow \mathcal{X} \rightarrow \mathbb{P}^2 \times \Delta \rightarrow \mathbb{P}^2$. For 
$d, m, a \in \mathbb{N}$ define a line bundle on $\mathcal{X}_1$ by
\[ \mathcal{L}_1 := f^\ast \mathcal{O}_{\mathbb{P}^2}(d) \otimes \mathcal{O}(-m\sum_{i=1}^{10}E_i) \otimes 
                              \mathcal{O}((2m+a)F_1). \]
Then $\mathcal{L}_{1|X_t} \cong \mathcal{L}(d; m^{10})$ for $t \neq 0$,
\[ \mathcal{L}_{1|P_1} \cong \mathcal{L}(2m+a; m^4), \]
where the $4$ points are in general position on $\mathcal{P}^2$, and
\[ \mathcal{L}_{1|F_1} \cong \mathcal{L}(d; 2m+a, m^6), \]
because $\mathcal{O}_{F_1}(F_1) \cong \mathcal{O}_{F_1}(-P_1)$, by the Three-point formula~\ref{three-lem} applied to
$X_{1,0} = P_1 \cup F_1$. The $7$ points on $\mathbb{P}^2$ can be assumed to lie in general position.

\noindent We assume $d > \sqrt{10}m$. To apply Theorem~\ref{Ineq-thm} we need 
$H^1(P_1, \mathcal{L}_{1|P_1}) = 0$, 
$H^1(F_1, \mathcal{L}_{1|F_1}) = 0$ and $H^1(F_1, \mathcal{L}_{1|F_1} \otimes \mathcal{O}(-C_1)) = 0$. 

\noindent For the vanishing on $P_1$ we choose an irreducible conic $C$ through the $4$ points blown up in $P_1$. The
strict transform of $C$ on $P_1$ is a section of $\mathcal{L}(2; 1^4)$. Since 
$\mathcal{L}(2; 1^4).(K_{P_1} \otimes \mathcal{L}(2; 1^4)) = \mathcal{L}(2; 1^4).\mathcal{L}(-1; 0^4) = -2 < 0$ and
$\mathcal{L}(2; 1^4).(\mathcal{L}_{1|P_1} \otimes \mathcal{O}(-iC)) = 2a \geq 0$, we can deduce 
$H^1(P_1, \mathcal{L}_{1|P_1}) = 0$ from 
$H^1(P_1, \mathcal{L}_{1|P_1} \otimes \mathcal{O}(-mC)) = H^1(P_1, \mathcal{L}(2a; 0^4)) = 0$, by Theorem~\ref{curvevan-crit}.

\noindent For the vanishing on $F_1$ we note first that $F_1$
is strongly anticanonical, as a blow up of $\mathbb{P}^2$ in less than $9$ points in general position. Next,
$\mathcal{L}_{1|F_1}.K_{F_1} = -3d + (2m + a) + 6m < 0$ for $a$ small enough. We perform Cremona
transformations on $\mathcal{L}_{1|F_1}$ changing the degree and multiplicities as follows:
\[ \begin{array}{cccc}
d; & \underline{2m+a}, &  & \underline{\underline{m}}^6\\
2d - 4m - a; & \underline{d - 2m}, & (d - 3m - a)^2, & \underline{\underline{m}}^4 \\
3d - 8m -2a; & \underline{2d - 6m - a}, & (d - 3m - a)^4, & \underline{\underline{m}}^2 \\
4d - 12m -3a; & 3d - 10m - 2a, & (d - 3m - a)^6. & 
\end{array} \]
Here, the underlinings indicate which $3$ points are used for the transformation.

\noindent After the Cremona transformations the intersection curve $C_1$ with $P_1$ on $F_1$ is a section of 
\[ \mathcal{L}(0;-1,0^6) \cong \mathcal{L}(1;0,1^2,0^4) \cong \mathcal{L}(2;1,1^4,0^2) \cong \mathcal{L}(3;2,1^6). \] 
Consequently, 
$\mathcal{L}_{1|F_1} \otimes \mathcal{O}(-C_1) \cong \mathcal{L}(4d - 12m -3a - 3; 3d - 10m - 2a - 2, (d - 3m - a - 1)^6)$. 
Both transformed line bundles are standard if the following inequalities are satisfied:
\[ \begin{array}{l}
4d - 12m -3a - 3 \geq 0,\ 3d - 10m - 2a - 2 \geq 0,\ d - 3m - a - 1 \geq 0, \\
4d - 12m - 3a  \geq  (3d - 10m - 2a) + 2(d - 3m - a) = 5d - 16m - 4a \\
4d - 12m - 3a  \geq  3(d - 3m - a) = 3d - 9m - 3a. 
\end{array} \]
The inequalities imply $d > \frac{10}{3} m$. On the other hand they are satisfied if $4m > d > \frac{10}{3} m$ and $a = 0$.
For such values of $d, m, a$ Harbourne's Criterion~\ref{Har-crit} implies the vanishing of the two $H^1$-groups. Consequently
we can apply Theorem~\ref{Ineq-thm} with $\mu = \frac{10}{3}$:
\begin{prop}
The multi-point Seshadri constant of $10$ points in general position is bounded from below by
\[ \epsilon(\mathbb{P}^2, \mathcal{O}_{\mathbb{P}^2}(1); p_1, \ldots, p_{10}) \geq \frac{3}{10}. \]
\end{prop}

\begin{rem}
We could also standardize the line bundle on $P_1$. But doing so we would loose the symmetry of the blown up points on
$P_1$ thus creating further difficulties when detecting curves to throw later on.
\end{rem}

\begin{rem} \label{mod-rem}
Tensorizing $f^\ast \mathcal{O}_{\mathbb{P}^2}(d) \otimes \mathcal{O}(-m\sum_{i=1}^{10}E_i)$ with $\mathcal{O}((2m+a)F_1)$
is necessary for providing enough positivity on the line bundle restricted to $P_1$. The multiple $2m$ would be the minimal
possible, but the additional $a$ helps in later degenerations. We will also use this type of modification again, to ensure
enough positivity for the line bundle on certain components. 
\end{rem}

\begin{rem} \label{genpos-rem}
The $4$ points on $P_p$ and the $6$ points on $F_p$ can be freely chosen. Any considerations on general position later on 
must backtrack to this choice. When transforming exceptional configurations this is done by the arguments in 
section~\ref{Crem-sec} without much effort. When discussing the intersection points of curves to throw with other components
we invert the Cremona transformation on the component containing the curve to throw, and argue on $P_p$ and $F_p$. 
\end{rem}

\subsection{The Second Degeneration} 

Still assuming $d > \sqrt{10} m$ we discuss what happens when $d < \frac{10}{3} m$.

\subsubsection{Identification of curves to throw} 

We look for curves to throw among the exceptional divisors associated to multiplicities of line bundles on
components of the last degeneration. These multiplicities must be negative when $d < \frac{10}{3} m$. This is the case for the
first multiplicity $3d - 10m - 2a$ of $\mathcal{L}_{1|F_1}$. Hence we want to throw the exceptional divisor $E_1 \subset F_1$
associated to this multiplicity. $E_1$ is a section of $\mathcal{L}(0;-1,0^6)$. 

\subsubsection{Intersection of curve to throw with other components}

Since 
\[ E_1.C_1 = \mathcal{L}(0;-1,0^6).\mathcal{L}(3;2,1^6) = 2, \] 
we expect two intersection points with $P_1$, and want to perform a $2$-throw. Since for $7$
points in general position on $\mathbb{P}^2$ the only section of $\mathcal{L}(3;2,1^6)$ is the strict transform of a cubic curve with a node in the first point, we indeed get $2$ different intersection points on the exceptional divisor over the node. On 
$P_1$ these $2$ points together with the $4$ blown up points can be assumed to lie in general position. 

\subsubsection{Throwing the curve: Components and their intersections}

In the Throwing Construction~\ref{throw-const} we identify $C$ with $E_1$, $V_1$ with $F_1$ and $V_2$ with $P_2$, and 
perform a $2$-throw. Call
\[ \mathcal{X}_2 := \widetilde{X},\ F_2 := \widetilde{V}_1,\ P_2 := \widetilde{V}_2, T_1^{(2)} := \widetilde{T}_1, \
   T_2^{(2)} := \widetilde{T}_2. \]
Then $F_2 \cong F_1$, $P_2 \cong P_1([p_1, p_2], [p_3, p_4])$ where $p_1, p_2, p_3, p_4$ all lie on the intersection curve
with $F_2$, $T_1^{(2)} \cong \mathbb{P}^2(p)$ and $T_2^{(2)} \cong \mathbb{P}^1 \times \mathbb{P}^1$.

\noindent Next, we describe the configuration of intersection curves on each component:
\begin{itemize}
\item on $F_2$: \begin{tabular}[t]{l}
                          (a section of) $\mathcal{L}(3;2,1^6)$ with $P_2$, \\
                          $\mathcal{L}(0;-1,0^6)$ with $T_2^{(2)}$,
                          \end{tabular}
\item on $P_2$: \begin{tabular}[t]{l}
                          $\mathcal{L}(1;0^4,[1,1],[1,1])$ with $F_2$, \\
                          $\mathcal{L}(0;0^4,[-1,1],[0,0])$ and $\mathcal{L}(0;0^4,[0,0],[-1,1])$ with $T_1^{(2)}$, \\
                          $\mathcal{L}(0;0^4,[0,-1],[0,0])$ and $\mathcal{L}(0;0^4,[0,0],[0,-1])$ with $T_2^{(2)}$,
                          \end{tabular}
\item on $T_1^{(2)}$: \begin{tabular}[t]{l}
                                 $2$ sections of $\mathcal{L}(1;1)$ with $P_2$, \\
                                 (the section of) $\mathcal{L}(0;-1)$ with $T_2^{(2)}$,
                                 \end{tabular}
\item on $T_2^{(2)}$: \begin{tabular}[t]{l}
                                 a (horizontal) section of $\mathcal{O}_{\mathbb{P}^1 \times \mathbb{P}^1}(0,1)$ with $F_2$ and with 
                                 $T_2^{(2)}$, \\
                                 $2$ (vertical) sections of $\mathcal{O}(1,0)$ with $P_2$.
                                 \end{tabular} 
\end{itemize}

\subsubsection{Throwing the curve: The line bundle and its restrictions}

In the Throwing Construction~\ref{throw-const} identify $\mathcal{L}$ with $\mathcal{L}_1$. Since 
$\mathcal{L}_1.\mathcal{E}_1 = \mathcal{L}_{1|F_1}.\mathcal{E}_1 = 3d - 10m - 2a$, we set
\[ a_1 := \frac{3}{2}d - 5m - a,\ a_2 := 3d - 10m - 2a \]
and only consider $d, m \in 2 \cdot \mathbb{N}$. Call $\mathcal{L}_2 := \widetilde{\mathcal{L}}$. Then
\begin{eqnarray*}
\mathcal{L}_{2|F_2} & \cong & \mathcal{L}(4d - 12m -3a; 0, (d - 3m - a)^6), \\
\mathcal{L}_{2|P_2} & \cong & \mathcal{L}(2m+a; m^4, [-a_1, -a_2+a_1]^2) \\
 & = & \mathcal{L}(2m+a; m^4, [a + 5m - \frac{3}{2}d, a + 5m - \frac{3}{2}d]^2), \\
\mathcal{L}_{2|T_1^{(2)}} & \cong & \mathcal{L}(a_2-2a_1; a_2 - 2a_1 - (a_2 - 2a_1)) = \mathcal{L}(0;0), \\
\mathcal{L}_{2|T_2^{(2)}} & \cong & \mathcal{O}_{\mathbb{P}^1 \times \mathbb{P}^1}(0,a_1-a_2) = 
                                                         \mathcal{O}(0, a + 5m - \frac{3}{2}d).
\end{eqnarray*}
Note that for $\sqrt{10}m < d < \frac{10}{3}m$ there are no negative multiplicities.

\subsubsection{Applying the Gluing Lemma}

In the setting of Gluing Lemma~\ref{glue-lem} we identify $V_1$ with $T_1^{(2)}$, $V_2$ with $T_2^{(2)}$, $V_3$ with $F_2$ and
$V_4$ with $P_2$. Then we check when the relevant cohomology groups vanish. 
\begin{enumerate}[leftmargin=*]
\item $H^1(T_1^{(2)}, \mathcal{L}_{2|T_1^{(2)}}) = 0$: obvious. 
\item $H^1(T_2^{(2)}, \mathcal{L}_{2|T_1^{(2)}}) = 0$ and 
         $H^1(T_2^{(2)}, \mathcal{L}_{2|T_1^{(2)}} \otimes \mathcal{O}(0,-1)) = 0$, for the intersection with $T_1^{(2)}$: true because 
         $a + 5m - \frac{3}{2}d > 0$.
\item $H^1(F_2, \mathcal{L}_{2|F_2}) = 0$ and $H^1(F_2, \mathcal{L}_{2|F_2} \otimes \mathcal{L}(0;1,0^6)) = 0$, for the
intersection with $T_2^{(2)}$: $F_2$ is strongly anti-canonical because it is the blow up of $\mathbb{P}^2$ in less than $9$
points in general position, by Remark~\ref{ac9-rem}. Since
\begin{eqnarray*} 
K_{F_2}.(\mathcal{L}_{2|F_2} \otimes \mathcal{L}(0;1,0^6)) & = & (-3) \cdot (4d - 12m - 3a) + 1 + 6(d - 3m - a)) \\ 
                                                                                              & = & -6d + 18m + 3a + 1. 
\end{eqnarray*}
is negative if $a < 2d - 6m$ and $\mathcal{L}(4d - 12m -3a; 1, (d - 3m - a)^6)$ is standard if $0 \leq a < d - 3m$
we can apply Harbourne's Criterion~\ref{Har-crit} if $0 \leq a < d - 3m$.
\item $H^1(P_2, \mathcal{L}_{2|P_2}) = 0$: $P_2$ is only anti-canonical because every section of $-K_{P_2}$ decomposes 
into the line $L$, a section of $\mathcal{L}(1;0^4,[1,1]^2)$, as the fixed part, and a conic $C$ in $\mathcal{L}(2;1^4,[0,0]^2)$ as
the moving part. We want to apply Theorem~\ref{curvevan-crit} using the curves $L$ and
$C$, but first we perform several Cremona transformations on $\mathcal{L}_{2|P_2}$:
\[ \begin{array}{lccl}
   2m+a; & \underline{\underline{\underline{m}}}^4, & & [5m-\frac{3}{2}d+a, 5m-\frac{3}{2}d+a]^2 \\
   m+2a; & \underline{m}, & a^3, & [\underline{\underline{5m - \frac{3}{2}d + a}}, 5m - \frac{3}{2}d + a]^2 
   \end{array} \]
Since the line $\mathcal{L}(1; 0^4, [1,1]^2)$ is transformed to a conic $\mathcal{L}(2; 0, 1^3, [1,1]^2)$ the infinitely near points
are tangent to this conic and not directed to one of the three base points of the next Cremona transformation indicated by
the underscores. We are in the setting of Cremona transformation II: 
\[ \begin{array}{lccl}
   3d - 9m + 2a; & \underline{3d - 9m}, & a^3, & (\frac{3}{2}d - 5m + a)^2, (\underline{\underline{5m - \frac{3}{2}d + a}})^2 \\
   6d-19m+2a; & 6d-19m, & \underline{\underline{\underline{a}}}^3, & [a-(5m - \frac{3}{2}d), a - (5m - \frac{3}{2}d)]^2 \\
   \end{array} \]
In particular, the non-infinitely near points remain in general position. A final Cremona transformation yields the more
symmetric configuration
\[    \mathcal{L}(12d-38m+a; (6d-19m)^4, [a-(5m - \frac{3}{2}d), a - (5m - \frac{3}{2}d)]^2). \]
After all these Cremona transformations, $L$ and $C$ are again sections of $\mathcal{L}(1; 0^4, [1,1]^2)$ and 
$\mathcal{L}(2; 1^4, [0,0]^2)$. Now, $(K_{P_2}+C).C = K_{P_2}.C = -2 < 0$, and 
$(\mathcal{L}_{2|P_2} - iC).C = \mathcal{L}_{2|P_2}.C = 2a \geq 0$. Hence $H^1(P_2, \mathcal{L}_{2|P_2}) = 0$ follows from
$H^1(P_2, \mathcal{L}(a;0^4,[a-(5m - \frac{3}{2}d), a - (5m - \frac{3}{2}d)]^2)) = 0$.

\noindent Next, $(K_{P_2}+L).L = -2 < 0$ and if $i > 0$ and $a < \frac{20}{3}m - 2d$,
\begin{eqnarray*}
(\mathcal{L}(a;0^4,[a-(5m - \frac{3}{2}d), a - (5m - \frac{3}{2}d)]^2) - iL).L & = & a - 4(a-(5m-\frac{3}{2}d)) + 3i \\
                                                                                                                  & > & 20m - 6d - 3a > 0. 
\end{eqnarray*}
Hence the vanishing of $H^1(P_2, \mathcal{L}(a;0^4,[a-(5m - \frac{3}{2}d), a - (5m - \frac{3}{2}d)]^2))$
follows from $H^1(P_2, \mathcal{L}(5m - \frac{3}{2}d; 0^8)) = 0$.
\item $H^1(P_2, \mathcal{L}_{2|P_2} \otimes \mathcal{O}(-C_2)) = 0$ where 
$C_2 = P_2 \cap (F_2 \cup T_1^{(2)} \cup T_2^{(2)})$: The Cremona transformations above do not change the description of 
intersection curves on $P_2$ with $F_2$ and $T_1^{(2)}$ whereas the $2$ intersection curves with $T_2^{(2)}$ become sections of $\mathcal{L}(2;1^4,[1,0],[0,0])$ and $\mathcal{L}(2;1^4,[0,0],[1,0])$. These curves add up to a section of 
$\mathcal{L}(5;2^4,[1,2]^2)$. As above we conclude that the first cohomology group of the resulting line bundle 
\[    \mathcal{L}(12d-38m+a-5; (6d-19m-2)^4, [a-(5m - \frac{3}{2}d) - 1, a - (5m - \frac{3}{2}d) - 2]^2). \]
vanishes if $H^1(P_2, \mathcal{L}(5m - \frac{3}{2}d+1; 0^4, [1,0]^2)) = 0$. Projecting from $P_2$ onto $\mathbb{P}^2$ blown up
in $2$ points we obtain an excellent line bundle on a strongly anti-canonical surface, hence the vanishing.
\end{enumerate}

\subsubsection{Bounds}

We can apply the Gluing Lemma~\ref{glue-lem} if the following inequalities are satisfied:
\[ \sqrt{10} m < d < \frac{3}{10} m,\ 0 \leq a < d - 3m,\ d > \frac{6}{19} m,\ a > 5m - \frac{3}{2} d,\ a < \frac{20}{3} m - 2d. \]
These inequalities imply $5m - \frac{3}{2} d < d - 3m \Leftrightarrow d > \frac{16}{5} m$. Vice versa they are satisfied if
\[ \frac{16}{5} m < d < \frac{29}{9} m\ \mathrm{and}\ 5m - \frac{3}{2} d < a < d - 3m, \]
because $d - 3m < \frac{20}{3} m - 2d \Leftrightarrow d < \frac{29}{9} m$. Consequently we can apply Theorem~\ref{Ineq-thm}
with $\mu = \frac{16}{5}$: 
\begin{prop}
The multi-point Seshadri constant of $10$ points in general position is bounded from below by
\[ \epsilon(\mathbb{P}^2, \mathcal{O}_{\mathbb{P}^2}(1); p_1, \ldots, p_{10}) \geq \frac{5}{16}. \]
\end{prop}

\subsection{The Third Degeneration} 

Still assuming $d > \sqrt{10} m$ we discuss what happens when $d < \frac{16}{5} m$. 

\subsubsection{Identification of curves to throw} 

This is more subtle than in the Second Degeneration: $d < \frac{16}{5} m$ implies that $5m - \frac{3}{2} d > d - 3m$, and we
cannot choose $a$ such that $5m - \frac{3}{2} d < a < d - 3m$. In the following we assume 
\[ a < d - 3m < 5m - \frac{3}{2} d. \]
Then multiplicities in the Cremona-transformed line bundle $\mathcal{L}_{2|P_2}$ become negative. Before identifying the
curves to throw we modify $\mathcal{L}_2$, for the reasons discussed in Remark~\ref{mod-rem}:
\[ \mathcal{L}^\prime_2 := \mathcal{L}_2 \otimes \mathcal{O}_{\mathcal{X}_2}((a - (5m - \frac{3}{2} d)T_1^{(2)}). \]
In the proof of Construction~\ref{throw-const} we showed
\[ \mathcal{O}_{T_1^{(2)}}(T_1^{(2)}) \cong \mathcal{O}_{\mathbb{F}_1}(-E_1 - 2F_1 - E_1) \cong \mathcal{L}(-2;0). \]
Furthermore, $\mathcal{O}_{T_2^{(2)}}(T_1^{(2)}) \cong \mathcal{O}_{\mathbb{P}^1 \times \mathbb{P}^1}(0,1)$,
$\mathcal{O}_{P_2}(T_1^{(2)}) \cong \mathcal{L}(0;0^4,[-1,1]^2)$ and $\mathcal{O}_{F_2}(T_1^{(2)}) \cong \mathcal{O}_{F_2}$.
Consequently,

\noindent
\begin{tabular}[t]{l}
$\mathcal{L}^\prime_{2|F_2} = \mathcal{L}_{2|F_2}$, 
$\mathcal{L}^\prime_{2|T_1^{(2)}} \cong \mathcal{L}(10m-3d-2a;0)$, 
$\mathcal{L}^\prime_{2|T_2^{(2)}} \cong \mathcal{O}_{\mathbb{P}^1 \times \mathbb{P}^1}(0,2a)$ and\\
$\mathcal{L}^\prime_{2|P_2} \cong \mathcal{L}(12d-38m+a; (6d-19m)^4, [0, 2a - 10m + 3d)]^2)$.
\end{tabular} 

\noindent We throw the two $(-1)$-curves $E_{2,1}$ of $\mathcal{L}(0;0^4,[0,-1],[0,0])$ and 
$E_{2,2}$ of $\mathcal{L}(0;0^4,[0,0],[0,-1])$ simultaneously. This is possible because they do not intersect on $P_2$:
\[ \mathcal{L}(0;0^4,[0,-1],[0,0]).\mathcal{L}(0;0^4,[0,0],[0,-1]) = 0. \]

\subsubsection{Intersection of curves to throw with other components}

The intersection of $E_{2,i}$with the other components can be computed on $P_2$, using the intersection curves of the other
components with $P_2$:
\begin{itemize}[leftmargin=*]
\item With $F_2$, there exists for both curves exactly
\begin{eqnarray*} 
\lefteqn{\mathcal{L}(1;0^4,[1,1],[1,1]).\mathcal{L}(0;0^4,[0,-1],[0,0]) =} \\
 & =  & \mathcal{L}(1;0^4,[1,1],[1,1]).\mathcal{L}(0;0^4,[0,0],[0,-1]) = 1 
\end{eqnarray*} 
intersection point. On $P_2$ these two points $p_1, p_2$ lie on the intersection curve $C_1 = P_2 \cap F_2$,
a section of $\mathcal{L}(1;0^4,[1,1]^2)$. Backtracking through the Cremona transformations on $P_2$ it is still a section of 
$\mathcal{L}(1;0^4,[1,1]^2)$, hence the (strict transform of the) line through the $2$ intersection points with the curve on 
$F_2$ thrown in the Second Degeneration. On the other hand, $E_{2,1}$ and $E_{2,2}$ become sections of
$\mathcal{L}(2;1^4,[1,0],[0,0])$ and $\mathcal{L}(2;1^4,[0,0],[1,0])$. Hence the second intersection point of these conics with the
line varies freely on the line when varying the $4$ points on $P_2$. Consequently, the points $p_1, p_2$
are in general position on $C_1$, in particular with respect to the 7 points blown up on $F_2$ determining $C_1$ as a section 
of $\mathcal{L}(3;2,1^6)$. (See also Remark~\ref{genpos-rem}.)  
\item We easily calculate $E_{2,1}.T_2^{(2)} = E_{2,2}.T_2^{(2)} = 0$.
\item $P_2$ intersects $T_1^{(2)}$ in a section $C_1$ of $\mathcal{L}(0;0^4,[-1,1],[0,0])$ and a section $C_2$ of 
$\mathcal{L}(0;0^4,[0,0],[-1,1])$. We easily calculate $E_{2,i}.T_1^{(2)} = 1$, and $E_{2,i}$ only intersects $C_i$. Since the 
intersection points do not lie on $T_2^{(2)}$, they are not collinear with the point blown up on $T_1^{(2)}$, hence in general
position. 
\end{itemize}

\subsubsection{Throwing the curve: Components and their intersections} 

In the Throwing Construction~\ref{throw-const} we identify $E_{2,1}$ resp. $E_{2,2}$ with $E_1$, $P_2$ with $V_1$, $F_2$
with $V_2$, $T_1^{(2)}$ with $V_3$ and $T_2^{(2)}$ with $V_4$, and simultaneously perform two $2$-throws. Call
\[ \mathcal{X}_3 := \widetilde{\mathcal{X}},\ P_3 := \widetilde{V}_1,\ F_3 := \widetilde{V}_2,\ T_1^{(2,3)} := \widetilde{V}_3,\ 
   T_2^{(2,3)} := \widetilde{V}_4, \]
\[ T_{1,1}^{(3)} := \widetilde{T}_{1,1},\ T_{1,2}^{(3)} := \widetilde{T}_{1,2},\ 
   T_{2,1}^{(3)} := \widetilde{T}_{2,1},\ T_{2,2}^{(3)} := \widetilde{T}_{2,2}. \]
Then $P_3 \cong P_2$, $F_3 \cong F_2([p_1,p_2],[p_3,p_4])$, 
\[ T_1^{(2,3)} \cong T_1^{(2)}([p_1,p_2],[p_3,p_4]) \cong \mathbb{P}^2(p, [p_1,p_2],[p_3,p_4]), \] 
where $p, p_1, p_3$ are not collinear and the infinitely near points $p_2, p_4$ are directed to $p$, 
$T_2^{(2,3)} \cong T_2^{(2)}$, $T_{1,1}^{(3)} \cong T_{1,2}^{(3)} \cong \mathbb{F}_1 \cong \mathbb{P}^2(p)$, 
$T_{2,1}^{(3)} \cong T_{2,2}^{(3)} \cong \mathbb{P}^1 \times \mathbb{P}^1$.

\noindent Next, we describe the configuration of intersection curves on each component:
\begin{itemize}[leftmargin=*]
\item On $P_3$:
\begin{tabular}[t]{l}
(the only section of) $\mathcal{L}(1;0^4,[1,1],[1,1])$ with $F_3$, \\
$\mathcal{L}(0;0^4,[-1,1],[0,0])$ and $\mathcal{L}(0;0^4,[0,0],[-1,1])$ with $T_1^{(2,3)}$, \\
$\mathcal{L}(2;1^4,[1,0],[0,0])$ and $\mathcal{L}(2;1^4,[0,0],[1,0])$ with $T_2^{(2,3)}$, \\
no intersections with $T_{1,1}^{(3)}$ and $T_{1,2}^{(3)}$, \\
$\mathcal{L}(0;0^4,[0,-1],[0,0])$ resp. $\mathcal{L}(0;0^4,[0,0],[0,-1])$ with $T_{2,1}^{(3)}$ and $T_{2,2}^{(3)}$. 
\end{tabular}
\item on $F_3$: 
\begin{tabular}[t]{l}
$\mathcal{L}(3;2,1^6,[1,1],[1,1])$ with $P_3$, \\
no intersection with $T_1^{(2,3)}$, $\mathcal{L}(0;-1,0^6,[0,0],[0,0])$ with $T_2^{(2,3)}$, \\
$\mathcal{L}(0;0,0^6,[-1,1],[0,0])$ with $T_{1,1}^{(3)}$ and $\mathcal{L}(0;0,0^6,[0,0],[-1,1])$ with $T_{1,2}^{(3)}$, \\
$\mathcal{L}(0;0,0^6,[0,-1],[0,0])$ with $T_{2,1}^{(3)}$ and $\mathcal{L}(0;0,0^6,[0,0],[0,-1])$ with $T_{2,2}^{(3)}$. 
\end{tabular}
\item on $T_1^{(2,3)}$: 
\begin{tabular}[t]{l}
$\mathcal{L}(1;1,[1,1],[0,0])$ and $\mathcal{L}(1;1,[0,0],[1,1])$ with $P_3$, \\
no intersection with $F_3$, $\mathcal{L}(0;-1,[0,0],[0,0])$ with $T_2^{(2,3)}$, \\
$\mathcal{L}(0;0,[-1,1],[0,0])$ with $T_{1,1}^{(3)}$ and $\mathcal{L}(0;0,[0,0],[-1,1])$ with $T_{1,2}^{(3)}$, \\
$\mathcal{L}(0;0,[0,-1],[0,0])$ with $T_{2,1}^{(3)}$ and $\mathcal{L}(0;0,[0,0],[0,-1])$ with $T_{2,2}^{(3)}$. 
\end{tabular}
\item on $T_2^{(2,3)}$: 
\begin{tabular}[t]{l}
$\mathcal{O}(0,1)$ with $F_3$ and $T_1^{(2,3)}$, two sections of $\mathcal{O}(1,0)$ with $P_3$, \\
no intersection with $T_{1,i}^{(3)}$ and $T_{2,i}^{(3)}$, $i=1,2$.
\end{tabular}
\item on $T_{1,i}^{(3)}$: 
\begin{tabular}[t]{l}
$\mathcal{L}(1;1)$ with $F_3$ and $T_1^{(2,3)}$, $\mathcal{L}(0;-1)$ with $T_{2,i}^{(3)}$, \\
no intersection with $P_3$, $T_2^{(2,3)}$ and $T_{j,i}^{(3)}$, $j=1,2$.
\end{tabular}
\item on $T_{2,i}^{(3)}$: 
\begin{tabular}[t]{l}
$\mathcal{O}(0,1)$ with $F_3$ and $T_1^{(2,3)}$, $\mathcal{O}(1,0)$ with $P_3$ and $T_{1,i}^{(3)}$, \\
no intersection with $T_2^{(2,3)}$ and $T_{j,i}^{(3)}$, $j=1,2$.
\end{tabular}
\end{itemize}

\subsubsection{Throwing the curve: The line bundle and its restrictions}

In the Throwing Construction~\ref{throw-const} identify $\mathcal{L}$ with $\mathcal{L}_2^\prime$. Since 
$\mathcal{L}_2^\prime.\mathcal{E}_{2,i} = \mathcal{L}_{2|P_2}^\prime.\mathcal{E}_{2,i} = 3d - 10m + 2a$, we set
\[ a_1 := \frac{3}{2}d - 5m + a,\ a_2 := 3d - 10m + 2a \]
and only consider $d, m \in 2 \cdot \mathbb{N}$. Call $\mathcal{L}_3 := \widetilde{\mathcal{L}}$. Then
\begin{eqnarray*}
\mathcal{L}_{3|P_3} & \cong & \mathcal{L}(12d - 38m + a; (6d - 19m)^4, [0,0]^2), \\
\mathcal{L}_{3|F_3} & \cong & \mathcal{L}(4d-12m-3a; 0, (d-3m-a)^6, [5m - \frac{3}{2}d -a, 5m - \frac{3}{2}d - a]^2), \\
\mathcal{L}_{3|T_1^{(2,3)}} & \cong & \mathcal{L}(10m - 3d -2a; 0, [5m - \frac{3}{2}d - a, 5m - \frac{3}{2}d - a]^2), \\
\mathcal{L}_{3|T_2^{(2,3)}} & \cong & \mathcal{L}_{2|T_2^{(2)}}^\prime \cong \mathcal{O}(0, 2a), \\
\mathcal{L}_{3|T_{1,i}^{(3)}} & \cong & \mathcal{L}(0;0), \\
\mathcal{L}_{3|T_{2,i}^{(3)}} & \cong & \mathcal{O}(0,5m - \frac{3}{2}d - a).
\end{eqnarray*}

\subsubsection{Applying the Gluing Lemma} \label{glue3-ssec}

In the setting of Gluing Lemma~\ref{glue-lem} we identify $V_1$ with $T_{1,1}^{(3)} \cup T_{1,2}^{(3)}$, $V_2$ with $T_1^{(2,3)}$,
$V_3$ with $T_{2,1}^{(3)} \cup T_{2,2}^{(3)}$, $V_4$ with $T_2^{(2,3)}$,  $V_5$ with $F_3$ and
$V_6$ with $P_3$. Then we check when the relevant cohomology groups vanish. 
\begin{enumerate}[leftmargin=*]
\item $H^1(T_{1,i}^{(3)}, \mathcal{L}_{3|T_{1,i}^{(3)}}) = H^1(\mathbb{P}^2(p), \mathcal{L}(0;0)) = 0$: obvious. For the surjectivity
on $V_2 \cap W_1$ Prop.~\ref{negexc-prop} implies 
$H^1(\mathbb{P}^2(p), \mathcal{L}(-1;-1)) = H^1(\mathbb{P}^2, \mathcal{O}_{\mathbb{P}^2}(-1)) = 0$.

\item $H^1(T_1^{(2,3)}, \mathcal{L}_{3|T_1^{(2,3)}}) = 0$: First, $T_1^{(2,3)}$ is strongly anti-canonical since we can find a smooth cubic curve passing through a configuration of $5$ points as blown up on $T_1^{(2,3)}$. Using the Cremona
transformation I in section~\ref{Crem-sec},
\[ \mathcal{L}_{3|T_1^{(2,3)}} \cong  
   \mathcal{L}(10m - 3d -2a; 0, [\underline{\underline{5m - \frac{3}{2}d - a}}, \underline{5m - \frac{3}{2}d - a}]^2) \]
can be standardized to
\[ \mathcal{L}(5m - 3d - a; [0,0]^2, 5m - \frac{3}{2}d - a). \]
Then we can apply Harbourne's Criterion~\ref{Har-crit}.

\noindent After the Cremona transformation the intersection curves are sections of the following line bundles:

\begin{tabular}[t]{l}
$\mathcal{L}(0;[0,0],[-1,1],0)$ and $\mathcal{L}(1;[0,0],[1,1],1)$ with $P_3$, \\
no intersection with $F_3$, $\mathcal{L}(0;[0,0],[0,-1],0)$ with $T_2^{(2,3)}$, \\
$\mathcal{L}(0;[-1,1],[0,0],0)$ with $T_{1,1}^{(3)}$ and $\mathcal{L}(1;[1,1],[0,0],1)$ with $T_{1,2}^{(3)}$, \\
$\mathcal{L}(1;[1,0],[1,0],0)$ with $T_{2,1}^{(3)}$ and $\mathcal{L}(0;[0,0],[0,0],-1)$ with $T_{2,2}^{(3)}$. 
\end{tabular}
\item $H^1(T_{2,i}^{(3)}, \mathcal{L}_{3|T_{2,i}^{(3)}}) = 0$ and 
$H^1(T_{2,i}^{(3)}, \mathcal{L}_{3|T_{2,i}^{(3)}} \otimes \mathcal{O}_{T_{2,i}^{(3)}}(-T_{1,i}^{(3)} - T_1^{(2,3)})) = 0$, for the
intersection with $W_2$: true if $5m - \frac{3}{2}d > a$, because 
\[ \mathcal{O}_{T_{2,i}^{(3)}}(-T_{1,i}^{(3)} - T_1^{(2,3)}) \cong \mathcal{O}_{\mathbb{P}^1 \times \mathbb{P}^1}(-1,-1). \]
\item $H^1(T_2^{(2,3)}, \mathcal{L}_{3|T_2^{(2,3)}}) = 0$ and
$H^1(T_2^{(2,3)}, \mathcal{L}_{3|T_2^{(2,3)}} \otimes \mathcal{O}_{T_2^{(2,3)}}(- T_1^{(2,3)})) = 0$, for the intersection with 
$W_3$: true if $a>0$ because 
$\mathcal{O}_{T_2^{(2,3)}}(- T_1^{(2,3)}) \cong \mathcal{O}_{\mathbb{P}^1 \times \mathbb{P}^1}(0,-1)$.
\item $H^1(F_3, \mathcal{L}_{3|F_3}) = 0$ and $H^1(F_3, \mathcal{L}_{3|F_3} \otimes \mathcal{O}_{F_3}(-W_4) = 0$, for the 
intersection with $W_4$: First, 
\[ \mathcal{L}_{3|F_3} \cong \mathcal{L}(4d-12m-3a; 0, (d-3m-a)^6, [\underline{\underline{5m - \frac{3}{2}d -a}}, 
                                                                                                           \underline{5m - \frac{3}{2}d - a}]^2)               \]
is not standard if $d < \frac{54}{17} m$ because 
\[ 4d-12m-3a < 3 \cdot (5m - \frac{3}{2}d - a)\ \Leftrightarrow d < \frac{54}{17} m. \]
Note that $\frac{19}{6} < \frac{54}{17}$. 

\noindent Next, the infinitely near points are not directed to any of the other points. Hence we can perform Cremona
Transformation I without specifying the third point, and obtain
\[ \mathcal{L}(\frac{25}{2}d-39m-3a; 0,(d-3m-a)^6, [7d-22m-a,7d-22m-a],7d-22m-a, 5m-\frac{3}{2}d-a) \]
Note that now the infinitely near point is directed to the last point blown up.
 
\noindent Under our assumption on $a$, this line bundle is standard if $\frac{19}{6} m < d < \frac{54}{17} m$ because then
$d-3m-a < 7d-22m-a < 5m-\frac{3}{2}d-a$, and
\[ \frac{25}{2}d-39m-3a = 2 \cdot (7d-22m-a) + (5m-\frac{3}{2}d-a). \]
After the Cremona transformation the intersection curves of $F_3$ with the other components are sections of the following line
bundles:

\begin{tabular}[t]{l}
$\mathcal{L}(3;2,1^6,[1,1],1,1)$ with $P_3$, no intersection with $T_1^{(2,3)}$, \\
$\mathcal{L}(0;-1,0^6,[0,0],0,0)$ with $T_2^{(2,3)}$, \\
$\mathcal{L}(0;0,0^6,[-1,1],0,0)$ with $T_{1,1}^{(3)}$ and $\mathcal{L}(1;0,0^6,[1,1],0,1])$ with $T_{1,2}^{(3)}$, \\
$\mathcal{L}(1;0,0^6,[1,0],1,0)$ with $T_{2,1}^{(3)}$ and $\mathcal{L}(0;0,0^6,[0,0],0,-1)$ with $T_{2,2}^{(3)}$. 
\end{tabular}

\noindent When showing $H^1(F_3, \mathcal{L}_{3|F_3}) = 0$ we can forget the point with multiplicity $0$ and study the line 
bundle 
 \[ \widetilde{\mathcal{L}} := \mathcal{L}(\frac{25}{2}d-39m-3a; (d-3m-a)^6, [7d-22m-a,7d-22m-a],7d-22m-a, 5m-\frac{3}{2}d-a) \]
on 
$\widetilde{F} = \mathbb{P}^2(p_1, \ldots, p_6, [p_7,p_8],p_9,p_{10})$. $\widetilde{F}$ is strongly anti-canonical because the
image of the cubic in $\mathcal{L}(3;2,1^6,[1,1],1,1)$ on $F_3$ is a section of $-K_{\widetilde{F}}$. Furthermore,
\begin{eqnarray*}
\widetilde{\mathcal{L}}.K_{\widetilde{F}} & = & -3(\frac{25}{2}d-39m-3a) + 6(d-3m-a) + 3(7d-22m-a) + (5m-\frac{3}{2}d-a) \\
  & = & -12d + 38m -a < 0 
\end{eqnarray*}
if $d > \frac{19}{6} m$. Consequently we can apply Harbourne's Criterion~\ref{Har-crit}.

\noindent Finally all the intersection curves of $F_3$ with components of $W_4$ add up to $\mathcal{L}(2;-1,0^6,[1,2],1,0)$.
Since $\mathcal{L}^\prime := \mathcal{L}_{3|F_3} \otimes \mathcal{L}(-2;1,0^6,[-1,-2],-1,0)$ has no vanishing multiplicity we 
cannot
work directly on $\widetilde{F}$. But we can apply Theorem~\ref{curvevan-crit} on $\mathcal{L}^\prime$ and the strict transform
$C$ of the cubic in $\mathcal{L}(3;2,1^6,[1,1],1,1)$, because $(K_{F_3}+C).C=-2$ and
\begin{eqnarray*} 
\mathcal{L}^\prime.C & = & 12d - 38m + a - 4 > 0 
\end{eqnarray*}
if $d > \frac{19}{6} m$ (and $a > 2$). Consequently we only have to show  
\[ H^1(F_3, \mathcal{L}_{3|F_3} \otimes \mathcal{L}(-5;-1,(-1)^6,[-2,-3],-2,-1)) = 0, \]
and after applying Prop.~\ref{negexc-prop} this follows as above, working on $\widetilde{F}$.

\item $H^1(P_3, \mathcal{L}_{3|P_3}) = 0$ and $H^1(P_3, \mathcal{L}_{3|P_3} \otimes \mathcal{O}_{P_3}(-W_5) = 0$, for the 
intersection with $W_5$: We can forget the points with multiplicity $0$ and work on 
\[ \widetilde{P} \cong \mathbb{P}^2(p_1, \ldots, p_4),\ \widetilde{\mathcal{L}} := \mathcal{L}(12d-38m+a; (6d-19m)^4). \]
$\widetilde{P}$ is strongly anti-canonical. Since the points $p_1, \ldots, p_4$ are not collinear we can perform a Cremona 
transformation on $3$ of them and obtain
\[ \widetilde{\mathcal{L}} \cong \mathcal{L}(6d-19m+2a; 6d-19m, a^3). \]
This is a standard line bundle, and $\widetilde{\mathcal{L}}.K_{\widetilde{P}} = -12d+38m-3a < 0$ if 
$d > \frac{19}{6} m$ and $a < 6d-19m$. Hence we can apply Harbourne's Criterion~\ref{Har-crit}.

\noindent Finally, the sum of all intersection curves of $P_3$ with components of $W_5$ is a section of 
$\mathcal{L}(3;1^4,[1,1]^2)$. By Prop.~\ref{negexc-prop}, 
\[ H^1(P_3, \mathcal{L}_{3|P_3} \otimes \mathcal{L}(-3;(-1)^4,[-1,-1]^2)) = 
   H^1(\widetilde{P}, \widetilde{\mathcal{L}} \otimes \mathcal{L}(-3;(-1)^4)). \]
We can standardize as above and apply Harbourne's Criterion~\ref{Har-crit}.

\noindent We do not use the above Cremona transformation in later degenerations.
\end{enumerate}

\subsubsection{Bounds}

We can apply the Gluing Lemma~\ref{glue-lem} if the following inequalities are satisfied:
\[ a < 6d - 19m,\ \frac{19}{6} m < d < \frac{54}{17} m. \]
Consequently we can apply Theorem~\ref{Ineq-thm}
with $\mu = \frac{19}{6}$: 
\begin{prop}
The multi-point Seshadri constant of $10$ points in general position is bounded from below by
\[ \epsilon(\mathbb{P}^2, \mathcal{O}_{\mathbb{P}^2}(1); p_1, \ldots, p_{10}) \geq \frac{6}{19}. \]
\end{prop}

\subsection{The Fourth Degeneration} 
Still assuming $d > \sqrt{10} m$ we discuss what happens when $d < \frac{19}{6} m$. 

\subsubsection{Identification of curves to throw} 

If $d < \frac{19}{6} m$ the $(-1)$-curves $E_{3,1}$, $E_{3,2}$, $E_{3,3}$, $E_{3,4}$ in
\[ \mathcal{L}(0;-1,0^3,[0,0]^2),\ \mathcal{L}(0;0,-1,0^2,[0,0]^2),\ \mathcal{L}(0;0^2,-1,0,[0,0]^2),\ \mathcal{L}(0;0^3,-1,[0,0]^2) \] intersect
$\mathcal{L}_{3|P_3} \cong \mathcal{L}(12d-38m+a; (6d-19m)^4)$ negatively, and they do not intersect each other. 

\noindent We want to throw simultaneously the $4$ curves $E_{3,1}, E_{3,2}, E_{3,3}, E_{3,4}$.

\subsubsection{Intersection of curves to throw with other components}

There is no intersection of $E_{3,1}$, $E_{3,2}$, $E_{3,3}$, $E_{3,4}$ with $F_3$, $T_1^{(2,3)}$, $T_{1,i}^{(3)}$, $T_{2,i}^{(3)}$.

\noindent The intersection curve of $P_3$ with $T_2^{(2,3)}$ consists of a conic $C_1$ in $\mathcal{L}(2;1^4,[1,0],[0,0])$ and a conic $C_2$ in $\mathcal{L}(2;1^4,[0,0],[1,0])$. Both sections
intersect each of the $E_{3,j}$ exactly once. Call the intersection points $p_j$ and $p_{4+j}$, $j=1, \ldots, 4$. 

\noindent None of the intersection points on $C_2$ lies on the same horizontal fiber of 
$T_2^{(2,3)} \cong \mathbb{P}^1 \times \mathbb{P}^1$ as one of the intersection points with $C_1$: Reversing the
Cremona transformations applied on $P_3 \cong P_2$ in the second degeneration, it turns out that the $E_{3,i}$ can also be
interpreted as quartics in 
\[ \mathcal{L}(4;1, 2^3,[1,1]^2),\ \mathcal{L}(4;2,1,2^2,[1,1]^2),\ \mathcal{L}(4;2^2,1,2,[1,1]^2),\ 
   \mathcal{L}(4;2^3,1,[1,1]^2). \]
The intersection curves $C_1$ and $C_2$ turn into sections of $\mathcal{L}(0;0^4,[0,-1],[0,0])$ and 
$\mathcal{L}(0;0^4,[0,0],[0,-1])$. These $(-1)$-curves are identified by the horizontal projection of $T_2^{(2,3)}$ on 
$\mathbb{P}^1$, and the identification is not
affected by different choices of the $4$ points blown up on $P_2$. On the other hand, moving the $4$ points with a pulled back
$\mathbb{C}^\ast$-action fixing all points on $C_1$ and only $2$ points on $C_2$ varies the quartic $E_{3,i}$ in such a way that the intersection points with $C_1$ are fixed, and those with $C_2$ vary.

\subsubsection{Throwing the curve: Components and their intersections} 

In the Throwing Construction~\ref{throw-const} we identify the curves $E_{3,j}$ with $E_1$, $P_3$ with $V_1$, $T_2^{(2,3)}$
with $V_2$, $F_3$, $T_1^{(2,3)}$, $T_{1,i}^{(3)}$, $T_{2,i}^{(3)}$ with $V_3, \ldots, V_8$, and simultaneously perform four 
$2$-throws. Call
\[ \mathcal{X}_4 := \widetilde{\mathcal{X}},\ P_4 := \widetilde{V}_1,\ T_2^{(2,4)} := \widetilde{V}_2, F_4, T_1^{(2,3)}, 
   T_{1,i}^{(3,4)}, T_{2,i}^{(3,4)}  := \widetilde{V}_3, \ldots, \widetilde{V}_8. \] 
Then $P_4 \cong P_3$, $F_4, T_1^{(2,4)}, T_{1,i}^{(3,4)}, T_{2,i}^{(3,4)}$ are isomorphic to 
$F_3, T_1^{(2,4)}, T_{1,i}^{(3)}, T_{2,i}^{(3)}$, and 
\[   T_2^{(2,4)} \cong T_2^{(2,3)}([p_1,q_1], \ldots, [p_8,q_8]) \] 
where $p_1, q_1, \ldots, p_4, q_4$ are on one vertical fiber, $p_5, q_5, \ldots, p_8, q_8$ are on another vertical fiber, and no
$2$ points $p_i, p_j$ are on the same horizontal fiber. Finally, 
\[ T_{1,j}^{(4)} \cong \mathbb{F}_1,\ T_{2,j}^{(4)} \cong \mathbb{P}^1 \times \mathbb{P}^1. \]

\noindent Next, we describe the configuration of intersection curves on each component. 
\begin{itemize}[leftmargin=*]
\item On $F_4, T_1^{(2,4)}, T_{1,i}^{(3,4)}, T_{2,i}^{(3,4)}$: as on $F_3, T_1^{(2,4)}, T_{1,i}^{(3)}, T_{2,i}^{(3)}$ in the Third
Degeneration.
\item On $P_4$:
\begin{tabular}[t]{l}
with $F_4, T_1^{(2,4)}, T_{1,i}^{(3,4)}, T_{2,i}^{(3,4)}$ as on $P_3$ with $F_3, T_1^{(2,4)}, T_{1,i}^{(3)}, T_{2,i}^{(3)}$ in the \\ 
Third Degeneration, \\
with $T_2^{(2,4)}$ as on $P_3$ with $T_2^{(2,3)}$, no intersections with $T_{1,j}^{(4)}$, \\
a section of $\mathcal{L}(0;-1,0,0,0,[0,0]^2)$, that is $E_{3,1}$, with $T_{2,1}^{(4)}$, \\
similarly with the other $T_{2,j}^{(4)}$.
\end{tabular}
\item On $T_2^{(2,4)}$: 
\begin{tabular}[t]{l}
$\mathcal{O}(1,0)([0,0]^4,[1,1]^4)$ and $\mathcal{O}(1,0)([1,1]^4,[0,0]^4)$ with $P_4$, \\
$\mathcal{O}(0,1)([0,0]^8)$ with $F_4$ and $T_1^{(2,4)}$, \\
$\mathcal{O}(0,0)([-1,1], [0,0]^3,[0,0]^4)$ and $\mathcal{O}(0,0)([0,0]^4,[-1,1],[0,0]^3)$ with $T_{1,1}^{(4)}$, \\
similarly with $T_{1,j}^{(4)}$, $j=2,3,4$, \\
$\mathcal{O}(0,0)([0,-1], [0,0]^3,[0,0]^4)$ and $\mathcal{O}(0,0)([0,0]^4,[0,-1],[0,0]^3)$ with $T_{2,1}^{(4)}$, \\ 
similarly with $T_{2,j}^{(4)}$, $j=2,3,4$, \\
\end{tabular}
\item On $T_{1,j}^{(4)}$: $2$ sections of $\mathcal{L}(1;1)$ with $T_2^{(2,4)}$, and one of $\mathcal{L}(0;-1)$ with 
         $T_{2,j}^{(4)}$.
\item On $T_{2,j}^{(4)}$: $2$ sections of $\mathcal{O}(1,0)$ with $T_2^{(2,4)}$, and one of $\mathcal{O}(0,1)$ with $P_4$,
         $T_{1,j}^{(4)}$.
\end{itemize}

\subsubsection{Throwing the curve: The line bundle and its restrictions}

In the Throwing Construction~\ref{throw-const} identify $\mathcal{L}$ with $\mathcal{L}_3$. Since 
$\mathcal{L}_3.\mathcal{E}_{3,j} = \mathcal{L}_{3|P_3}.\mathcal{E}_{3,j} = 6d - 19m$, we set
\[ a_1 := 3d - \frac{19}{2}m,\ a_2 := 6d - 19m \]
and only consider $d, m \in 2 \cdot \mathbb{N}$. Call $\mathcal{L}_4 := \widetilde{\mathcal{L}}$. Then
\[ \mathcal{L}_{4|P_4} \cong  \mathcal{L}_{3|P_3} \otimes \mathcal{L}(0;(-1)^4,[0,0]^2)^{\otimes a_2} \cong 
   \mathcal{L}(12d - 38m + a; 0^4, [0,0]^2), \]
\[ \mathcal{L}_{4|F_4} \cong \mathcal{L}_{3|F_3},\  \mathcal{L}_{4|T_1^{(2,4)}} \cong \mathcal{L}_{3|T_1^{(2,3)}},  
   \mathcal{L}_{4|T_{1,i}^{(3,4)}} \cong \mathcal{L}_{3|T_{1,i}^{(3)}},\ 
   \mathcal{L}_{4|T_{2,i}^{(3,4)}} \cong \mathcal{L}_{3|T_{2,i}^{(3)}}, \]
\[ \mathcal{L}_{4|T_{1,j}^{(4)}} \cong \mathcal{L}(0;0),\ \mathcal{L}_{4|T_{2,j}^{(4)}} \cong \mathcal{O}(0, \frac{19}{2}m-3d),\ 
   j=1,\ldots,4, \]
\[ \mathcal{L}_{4|T_2^{(2,4)}} \cong \mathcal{O}(0, 2a)([\frac{19}{2}m-3d, \frac{19}{2}m-3d]^8). \]

\subsubsection{Applying the Gluing Lemma not possible}

Consider the strict transforms $E_{4,k}$, $k=1,\ldots,8$ of horizontal fibers through one of the $8$ points blown up on 
$T_2^{(2,4)}$. The $E_{4,k}$ are sections of $\mathcal{O}(0,1)([0,0]^{k-1},[1,0],[0,0]^{8-k})$, and 
\[ E_{4,k}.\mathcal{L}_{4|T_2^{(2,4)}} = - (\frac{19}{2}m-3d) \leq -2 \]
if $6d \leq -4 + 19m$. Lemma~\ref{spec-lem} implies that for these $d,m$ the line bundle $\mathcal{L}_{4|T_2^{(2,4)}}$ is
special.

\noindent Consequently, we cannot apply the Gluing Lemma, and we must perform further throws to obtain new bounds for the
Seshadri constant. 
 
\subsection{The Fifth Degeneration} 

Without changing the assumption $\sqrt{10} m < d < \frac{19}{6} m$ we want to throw the $8$ curves $E_{4,k}$ on 
$T_2^{(2,4)} \subset \mathcal{X}_4$ simultaneously. This is possible because the $E_{4,k}$ are pairwise disjoint. Before
throwing them we modify the line bundle $\mathcal{L}_4$, for the reasons discussed in Remark~\ref{mod-rem}: 
\[ \mathcal{L}_4^\prime := \mathcal{L}_4 \otimes \mathcal{O}_{\mathcal{X}_4}(- (\frac{19}{2}m - 3d) \sum_{j=1}^4 T_{1,j}^{(4)}). \]
The restrictions of $\mathcal{L}_4^\prime$ are the same as those of 
$\mathcal{L}_4$ on all components besides those intersecting one of the $T_{1,j}^{(4)}$, that is $T_{2,j}^{(4)}$, $T_2^{(2,4)}$
and $T_{1,j}^{(4)}$ itself. 

\noindent In the proof of Construction~\ref{throw-const} we showed
\[ \mathcal{O}_{T_{1,j}^{(4)}}(T_{1,j}^{(4)}) \cong \mathcal{O}_{\mathbb{F}_1}(-E_1 - 2F_1 - E_1) \cong \mathcal{L}(-2;0). \]
Furthermore, 
\[ \mathcal{O}_{T_{2,j}^{(4)}}(T_{1,j}^{(4)}) \cong \mathcal{O}_{\mathbb{P}^1 \times \mathbb{P}^1}(0,1),\ 
   \mathcal{O}_{T_2^{(2,4)}}(\sum_{j=1}^4 T_{1,j}^{(4)}) \cong \mathcal{O}(0,0)([-1,1]^8). \]
Consequently,
\[ \mathcal{L}^\prime_{4|T_{1,j}^{(4)}} = \mathcal{L}(19m - 6d;0),\ \mathcal{L}^\prime_{4|T_{2,j}^{(4)}} = \mathcal{O}(0,0),\ 
   \mathcal{L}^\prime_{4|T_2^{(2,4)}} \cong \mathcal{O}(0,2a)([19m-6d,0]^8). \]

\subsubsection{Intersection of curves to throw with other components}

The curves $E_{4,1}, \ldots, E_{4,4}$ intersect $P_4$ in exactly one point, on the component not containing 
$p_1, \ldots, p_4$. The curves $E_{4,5}, \ldots, E_{4,8}$ intersect $P_4$ in exactly one point, on the component not containing 
$p_5, \ldots, p_8$. Finally, each of the $E_{4,k}$ intersects $T_{1,j}^{(4)}$ in exactly one point iff $k \equiv j \mod 4$.

\subsubsection{Throwing the curve: Components and their intersections} 

In the Throwing Construction~\ref{throw-const} we identify the curves $E_{4,k}$ with $E_1$, $T_2^{(2,4)}$ with $V_1$, $P_4$
with $V_2$, $T_{1,j}^{(4)}$, $j=1,\ldots,4$ with $V_3, \ldots, V_6$, 
$F_4$, $T_1^{(2,4)}$, $T_{1,i}^{(3,4)}$, $T_{2,i}^{(3,4)}$, $i=1,2$, $T_{2,j}^{(4)}$, $j=1,\ldots,4$ with $V_7, \ldots, V_{16}$, and simultaneously perform eight $2$-throws. Call
\[ \mathcal{X}_5 := \widetilde{\mathcal{X}},\ T_2^{(2,5)} := \widetilde{V}_1,\ P_5 := \widetilde{V}_2,\ 
   T_{1,j}^{(4,5)} := \widetilde{V}_{2+j},\ j=1,\ldots,4, \] 
\[ F_5, T_1^{(2,5)}, T_{1,i}^{(3,5)}, T_{2,i}^{(3,5)}, T_{2,j}^{(4,5)}   := \widetilde{V}_7, \ldots, \widetilde{V}_{16}, \]
\[ T_{1,k}^{(5)} := \widetilde{T}_{1,k},\ T_{2,k}^{(5)} := \widetilde{T}_{2,k},\ k=1,\ldots,8. \] 
Then $T_2^{(2,5)} \cong T_2^{(2,4)}$, $P_5 \cong P_4([q_1,q_1^\prime], \ldots, [q_8,q_8^\prime])$, where $q_1, \ldots, q_4$
lie on the first intersection curve of $P^4$ with $T_2^{(2,4)}$, and $q_5, \ldots, q_8$
on the second intersection curve, $T_{1,j}^{(4,5)} \cong T_{1,j}^{(4)}([q_{1,j}, q_{1,j}^\prime],[q_{1,4+j}, q_{1,4+j}^\prime])$, 
$T_{1,k}^{(5)} \cong \mathbb{F}_1$, $T_{2,k}^{(5)} \cong \mathbb{P}^1 \times \mathbb{P}^1$, $k=1,\ldots,8$, 
and finally $\widetilde{V}_l \cong V_l$, $l=7,\ldots,16$.

\noindent Next, we describe the configuration of intersection curves on each component. 
\begin{itemize}[leftmargin=*]
\item On $F_5, T_1^{(2,5)}, T_{1,i}^{(3,5)}, T_{2,i}^{(3,5)}, T_{2,j}^{(4,5)}$ as on $F_4$, $T_1^{(2,4)}$, $T_{1,i}^{(3,4)}$, 
$T_{2,i}^{(3,4)}$, $T_{2,j}^{(4)}$ in the Fourth Degeneration.
\item On $T_2^{(2,5)}$: 
\begin{tabular}[t]{l}
with $P_5, F_5, T_1^{(2,5)}, T_{1,j}^{(4,5)}, T_{2,j}^{(4,5)}$ as with $P_4, F_4, T_1^{(2,4)}, T_{1,j}^{(4)}, T_{2,j}^{(4)}$ in the\\
Fourth Degeneration, with $T_{2,k}^{(5)}$: $E_{5,k}$, $k=1,\ldots,8$.
\end{tabular}
\item On $P_5$: 
\begin{tabular}[t]{l}
with $F_5, T_1^{(2,5)}, T_{1,i}^{(3,5)}, T_{2,i}^{(3,5)},T_{1,j}^{(4,5)},  T_{2,j}^{(4,5)}$ as with $F_4$, $T_1^{(2,4)}$, 
$T_{1,i}^{(3,4)}$, $T_{2,i}^{(3,4)}$, \\
$T_{1,j}^{(4)}$, $T_{2,j}^{(4)}$ in the Fourth Degeneration, \\
with $T_2^{(2,5)}$: $\mathcal{L}(2;1^4,[1,0],[0,0],[1,1]^4,[0,0]^4)$ and \\
$\mathcal{L}(2;1^4,[0,0],[1,0],[0,0]^4,[1,1]^4)$, \\
with $T_{1,k}^{(5)}$: $\mathcal{L}(0;0^4,[0,0]^2,[1,1]^4,[0,0]^{k-1},[-1,1],[0,0]^{8-k})$, \\  
with $T_{2,k}^{(5)}$: $\mathcal{L}(0;0^4,[0,0]^2,[1,1]^4,[0,0]^{k-1},[0,-1],[0,0]^{8-k})$.
\end{tabular}
\item On $T_{1,j}^{(4,5)}$: 
\begin{tabular}[t]{l}
with $T_{2,j}^{(4,5)}$ as with $T_{2,j}^{(4)}$ in the Fourth Degeneration, \\
with $T_{1,j}^{(5)}$: $\mathcal{L}(0;0,[-1,1],[0,0])$, with $T_{1,j+4}^{(5)}$: $\mathcal{L}(0;0,[0,0],[-1,1])$, \\
with $T_{2,j}^{(5)}$: $\mathcal{L}(0;0,[0,-1],[0,0])$, with $T_{2,j+4}^{(5)}$: $\mathcal{L}(0;0,[0,0],[0,-1])$, \\
with $T_2^{(2,5)}$: $\mathcal{L}(1;1,[1,1],[0,0])$ and $\mathcal{L}(1;1,[0,0],[1,1])$.
\end{tabular}
\item On $T_{1,k}^{(5)}$: 
\begin{tabular}[t]{l}
with $P_5$: $\mathcal{L}(1;1)$, with $T_{1,j}^{(4,5)}$: $\mathcal{L}(1;1)$ if $j \equiv k \mod 4$, with $T_{1,k}^{(5)}$: 
$\mathcal{L}(0;-1)$
\end{tabular}
\item On $T_{2,k}^{(5)}$: 
\begin{tabular}[t]{l}
with $P_5$: $\mathcal{O}(0,1)$, with $T_{1,j}^{(4,5)}$: $\mathcal{O}(0,1)$ if $j \equiv k \mod 4$, \\
with $T_{2,k}^{(5)}$ and $T_2^{(2,5)}$: $\mathcal{O}(1,0)$.
\end{tabular}
\end{itemize}

\subsubsection{Throwing the curve: The line bundle and its restrictions}

In the Throwing Construction~\ref{throw-const} we identify $\mathcal{L}$ with $\mathcal{L}_4^\prime$. Since 
$\mathcal{L}_4^\prime.\mathcal{E}_{4,k} = \mathcal{L}_{4|P_4}^\prime.\mathcal{E}_{4,k} = 6d - 19m$, we set
\[ a_1 := 3d - \frac{19}{2}m,\ a_2 := 6d - 19m \]
and only consider $d, m \in 2 \cdot \mathbb{N}$. Call $\mathcal{L}_5 := \widetilde{\mathcal{L}}$. Then:
\[ \mathcal{L}_{5|F_5} \cong \mathcal{L}_{4|F_4},\ \mathcal{L}_{5|T_1^{(2,5)}} \cong \mathcal{L}_{4|T_1^{(2,4)}},\ 
   \mathcal{L}_{5|T_{1,i}^{(3,5)}} \cong \mathcal{L}_{4|T_{1,i}^{(3,4)}},\ 
   \mathcal{L}_{5|T_{2,i}^{(3,5)}} \cong \mathcal{L}_{4|T_{2,i}^{(3,4)}},\ i=1,2, \]
\[ \mathcal{L}_{5|T_{2,j}^{(4,5)}} \cong \mathcal{L}_{4|T_{2,j}^{(4)}} = \mathcal{O}(0,0),\ j=1,2,3,4, \]
\[ \mathcal{L}_{5|T_2^{(2,5)}} \cong \mathcal{O}(0,2a-8(19m-6d))), 
   \mathcal{L}_{5|P_5} \cong \mathcal{L}_{4|P_4}([\frac{19}{2}m-3d,\frac{19}{2}m-3d]^8),\ \]
\[ \mathcal{L}_{5|T_{1,j}^{(4,5)}} \cong \mathcal{L}(19m-6d;0, [\frac{19}{2}m-3d,\frac{19}{2}m-3d]^2),\ j=1,2,3,4, \]
\[ \mathcal{L}_{5|T_{1,k}^{(5)}} \cong \mathcal{L}(0;0),\ \mathcal{L}_{5|T_{1,k}^{(5)}} \cong \mathcal{O}(0,\frac{19}{2}m-3d),\ 
   k=1, \ldots, 8. \] 

\subsubsection{Applying the Gluing Lemma}

In the setting of Gluing Lemma~\ref{glue-lem} we identify $V_1$ with $T_{1,1}^{(3,5)} \cup T_{1,2}^{(3,5)}$, $V_2$ with 
$T_1^{(2,5)}$, $V_3$ with $T_{2,1}^{(3,5)} \cup T_{2,2}^{(3,5)}$, $V_4$ with $\bigcup_{j=1}^4 T_{2,j}^{(4,5)}$, $V_5$ with 
$T_2^{(2,5)}$,  $V_6$ with $\bigcup_{j=1}^4 T_{1,j}^{(4,5)}$, $V_7$ with $\bigcup_{k=1}^8 T_{2,k}^{(5)}$, $V_8$ with 
$\bigcup_{k=1}^8 T_{1,k}^{(5)}$, $V_9$ with $F_5$ and
$V_{10}$ with $P_5$. Then we check when the relevant cohomology groups vanish. 
\begin{enumerate}[leftmargin=*]
\item $H^1(T_{1,i}^{(3,5)}, \mathcal{L}_{5|T_{1,i}^{(3,5)}}) = 0$ and 
$H^1(T_{1,i}^{(3,5)}, \mathcal{L}_{5|T_{1,i}^{(3,5)}} \otimes \mathcal{O}_{T_{1,i}^{(3,5)}}(-T_1^{(2,5)})) = 0$, for the intersection
with $V_2$: true because $T_{1,i}^{(3,5)} \cong \mathbb{P}^2(p)$, $\mathcal{L}_{5|T_{1,i}^{(3,5)}} \cong \mathcal{L}(0;0)$ and
$\mathcal{O}_{T_{1,i}^{(3,5)}}(-T_1^{(2,5)}) \cong \mathcal{L}(-1;-1)$.

\item $H^1(T_1^{(2,5)}, \mathcal{L}_{5|T_1^{(2,5)}}) = 0$: Since 
$\mathcal{L}_{5|T_1^{(2,5)}} \cong \mathcal{L}(5m-\frac{3}{2}d-a; 5m-\frac{3}{2}d-a, [0,0]^2)$ this follows from Harbourne's
Criterion~\ref{Har-crit} if $a < 5m-\frac{3}{2}d$.

\item $H^1(T_{2,i}^{(3,5)}, \mathcal{L}_{5|2_{2,i}^{(3,5)}}) = 0$ and 
$H^1(T_{2,i}^{(3,5)}, \mathcal{L}_{5|T_{2,i}^{(3,5)}} \otimes \mathcal{O}_{T_{2,i}^{(3,5)}}(-W_2)) = 0$, for the intersection
with $W_2$: Since $\mathcal{L}_{5|T_{2,i}^{(3,5)}} \cong \mathcal{O}(0,5m-\frac{3}{2}d-a)$ and the intersection curves with
$W_2$ add up to a section of $\mathcal{O}(1,1)$ the vanishings follow if $a < 5m-\frac{3}{2}d$, using Prop.~\ref{negexc-prop}
for the second cohomology group.

\item $H^1(T_{2,j}^{(4,5)}, \mathcal{L}_{5|T_{2,j}^{(4,5)}}) = 0$: true because $\mathcal{L}_{5|T_{2,j}^{(4,5)}} =  \mathcal{O}(0,0)$.
Note that $V_4 \cap W_3 = \emptyset$.

\item $H^1(T_2^{(2,5)}, \mathcal{L}_{5|T_2^{(2,5)}}) = 0$ and 
$H^1(T_2^{(2,5)}, \mathcal{L}_{5|T_2^{(2,5)}} \otimes \mathcal{O}_{T_2^{(2,5)}}(-W_4)) = 0$, for the intersection
with $W_4$: Since $\mathcal{L}_{5|T_2^{(2,5)}} \cong \mathcal{O}(0, 2a-8(19m-6d))([0,0]^8)$, the first vanishing holds if 
$a \geq 4(19m-6d)$. 

\noindent For
$\mathcal{L}_{5|T_2^{(2,5)}} \otimes \mathcal{O}_{T_2^{(2,5)}}(-W_4) \cong \mathcal{O}(0, 2a-8(19m-6d)-1)([0,1]^8)$ we apply
Theorem~\ref{curvevan-crit} with $C = C_1 \cup C_2$,
\[ C_1\ \mathrm{section\ of\ } \mathcal{O}(0,1)([1,1]^4,[0,0]^4),\ 
                                 C_2\ \mathrm{section\ of\ } \mathcal{O}(0,1)([0,0]^4,[1,1]^4): \]
\[ \left[ K_{T_2^{(2,5)}} \otimes \mathcal{O}_{T_2^{(2,5)}}(C) \right].C_i = \mathcal{O}(-2,0)([0,0]^8).C_i = 0, \]
\[ \left[ \mathcal{L}_{5|T_2^{(2,5)}} \otimes \mathcal{O}_{T_2^{(2,5)}}(-W_4)) \right].C_i =  2a-8(19m-6d)-1-4 > 0 \]
if $a > 4(19m-6d)+2$, 
\begin{eqnarray*}
\lefteqn{H^1(T_2^{(2,5)}, \mathcal{L}_{5|T_2^{(2,5)}} \otimes \mathcal{O}_{T_2^{(2,5)}}(-W_4) \otimes 
                                       \mathcal{O}_{T_2^{(2,5)}}(-C)) =} \\ 
   & = & H^1(T_2^{(2,5)}, \mathcal{O}(0,2a-8(19m-6d)-1-2)([-1,0]^8) = 0, 
\end{eqnarray*}
if $a > 4(19m-6d)+2$, using Prop.~\ref{negexc-prop}. 

\item $H^1(T_{1,j}^{(4,5)}, \mathcal{L}_{5|T_{1,j}^{(4,5)}}) = 0$ and 
$H^1(T_{1,j}^{(4,5)}, \mathcal{L}_{5|T_{1,j}^{(4,5)}} \otimes \mathcal{O}_{T_{1,j}^{(4,5)}}(-W_5)) = 0$, for the intersection
with $W_5$: By Cremona Transformation I we can write 
$\mathcal{L}_{5|T_{1,j}^{(4,5)}} \cong \mathcal{L}(19m-6d;0, [\underline{\underline{\frac{19}{2}m-3d}},
                                                                                                 \underline{\frac{19}{2}m-3d}]^2)$ as
\[ \mathcal{L}_{5|T_{1,j}^{(4,5)}} \cong \mathcal{L}(\frac{19}{2}m-3d;[0,0]^2, \frac{19}{2}m-3d). \]
Similarly,
\[ \mathcal{L}_{5|T_{1,j}^{(4,5)}} \otimes \mathcal{O}_{T_{1,j}^{(4,5)}}(-W_5)) \cong 
   \mathcal{L}(\frac{19}{2}m-3d-1;[0,0],[0,-1],\frac{19}{2}m-3d-1) \]
because the intersection curves on $T_{1,j}^{(4,5)}$ with $W_5$ add up to $\mathcal{L}(2,1,[1,1]^2)$ which is written as 
$\mathcal{L}(1;[0,0],[0,1],1)$ after the Cremona transformation.

\noindent Then both vanishings follow from Harbourne's Criterion~\ref{Har-crit} and Prop.~\ref{negexc-prop} if 
$19m \geq 6d$.

\item $H^1(T_{2,k}^{(5)}, \mathcal{L}_{5|T_{2,k}^{(5)}}) = 0$ and 
$H^1(T_{2,k}^{(5)}, \mathcal{L}_{5|T_{2,k}^{(5)}} \otimes \mathcal{O}_{T_{2,k}^{(5)}}(-W_6)) = 0$, for the intersection
with $W_6$: Since $\mathcal{L}_{5|T_{2,k}^{(5)}} \cong \mathcal{O}(0,\frac{19}{2}m-3d)$ and the intersection curves with 
$W_6$ add up to a section of $\mathcal{O}(1,1)$, this is true if $\frac{19}{2}m-3d \geq 0$.

\item $H^1(T_{1,k}^{(5)}, \mathcal{L}_{5|T_{1,k}^{(5)}}) = 0$ and 
$H^1(T_{1,k}^{(5)}, \mathcal{L}_{5|T_{1,k}^{(5)}} \otimes \mathcal{O}_{T_{1,k}^{(5)}}(-W_7)) = 0$, for the intersection
with $W_7$: Since $\mathcal{L}_{5|T_{1,k}^{(5)}} \cong \mathcal{L}(0;0)$ and the intersection curves with 
$W_7$ add up to a section of $\mathcal{L}(1;0)$, this is true.

\item $H^1(F_5, \mathcal{L}_{5|F_5}) = 0$ and 
$H^1(F_5, \mathcal{L}_{5|F_5} \otimes \mathcal{O}_{F_5}(-W_8)) = 0$, for the intersection
with $W_8$: We start with four Cremona transformations on $\mathcal{L}_{5|F_5} \cong \mathcal{L}_{3|F_3}$:
\[ \begin{array}{lll}
   \frac{25}{2}d-39m-3a; & 0, & (\underline{\underline{d-3m-a}})^6, [7d-22m-a,7d-22m-a],  \\
                                      &     &                            7d-22m-a,  \underline{5m-\frac{3}{2}d-a} \\
   \frac{49}{2}d-77m-3a; & 0, & (13d-41m-a)^2,  (\underline{\underline{d-3m-a)}}^4,  [7d-22m-a,7d-22m-a], \\
                          & & 7d-22m-a, \underline{\frac{21}{2}d-33m-a}), \\
   \frac{73}{2}d-115m-3a; & 0, & (13d-41m-a)^4,  (\underline{\underline{d-3m-a)}}^2,  [7d-22m-a,7d-22m-a], \\
                          & & 7d-22m-a, \underline{\frac{45}{2}d-71m-a}), \\
   \frac{97}{2}d-153m-3a; & 0, & (13d-41m-a)^6, [\underline{7d-22m-a},\underline{7d-22m-a}], \\
                          & & \underline{7d-22m-a}, \frac{69}{2}d-109m-a), \\
   76d-240m-3a; & 0, & (13d-41m-a)^6,  [\frac{69}{2}d-109m-a, \frac{69}{2}d-109m-a]^2).
\end{array} \]

\noindent These transformations are possible because before and after the first three Cremona transformations the infinitely 
near point is directed to the third base point: this situation is described in Cremona transformation III.
In the last transformation, the last point blown up becomes infinitely near, as described in Cremona transformation I. 

\noindent  After the Cremona transformations the intersection curves of $F_5$ with the other components can be written as
sections of the following line bundles:

\begin{tabular}[t]{l}
$\mathcal{L}(3;2,1^6,[1,1],[1,1])$ with $P_5$, $\mathcal{L}(0;-1,0^6,[0,0],[0,0])$ with $T_2^{(2,5)}$, \\
no intersection with $T_1^{(2,5)}$, $T_{1,j}^{(4,5)}$, $T_{2,j}^{(4,5)}$, $T_{1,k}^{(5)}$, $T_{2,k}^{(5)}$, \\
$\mathcal{L}(0;0,0^6,[-1,1],[0,0])$ with $T_{1,1}^{(3,5)}$ and $\mathcal{L}(0;0,0^6,[0,0],[-1,1])$ with $T_{1,2}^{(3,5)}$, \\
$\mathcal{L}(6;0,1^6,[3,2],[3,3])$ with $T_{2,1}^{(3,5)}$ and $\mathcal{L}(6;0,1^6,[3,3],[3,2])$ with $T_{2,2}^{(3,5)}$. 
\end{tabular}

\noindent As in~\ref{glue3-ssec}(5) we can forget the point with multiplicity $0$ and study the line 
bundle 
 \[ \widetilde{\mathcal{L}} := \mathcal{L}(76d-240m-3a; (13d-41m-a)^6, [\frac{69}{2}d-109m-a, \frac{69}{2}d-109m-a]^2) \]
on 
$\widetilde{F} = \mathbb{P}^2(p_1, \ldots, p_6, [p_7,p_8],[p_9,p_{10}])$. As in~\ref{glue3-ssec}(5) the surface 
$\widetilde{F}$ is strongly anti-canonical, and
\[ \widetilde{\mathcal{L}}.K_{\widetilde{F}} = -12d + 38m -a < 0 \]
if $a > 2(19m - 6d)$. Finally, $\widetilde{\mathcal{L}}$ is standard if $d < \frac{136}{43}m$ and $a \leq \frac{69}{2}d-109m$:
Then, $0 \leq \frac{69}{2}d-109m-a < 13d-41m-a$,  $\frac{69}{2}d-109m-a < \frac{76}{3}d-80m$ and
\[  76d-240m-3a > 39d - 123m - 3a\ \Leftrightarrow 37d > 117m \]
holds because $\frac{117}{37} < \sqrt{10}$. Hence we can apply Harbourne's Criterion~\ref{Har-crit}.  

\noindent As in~\ref{glue3-ssec}(5) the intersection curves with $W_8$ add up to a section of $\mathcal{L}(2;-1,0^6,[1,2],1,0)$.
After the four Cremona transformations this line bundle can be written as
\[ \mathcal{O}_{F_5}(W_8) \cong \mathcal{L}(12;-1,2^6,[5,6],[5,6]). \]
We want to argue as in~\ref{glue3-ssec}(5) and apply Theorem~\ref{curvevan-crit} on 
\[ \mathcal{L}^\prime := \mathcal{L}_{5|F_5} \otimes \mathcal{L}(-12;1,(-2)^6,[-5,-6],[-5,-6]) \]
and the cubic $C$ in $\mathcal{L}(3;2,1^6,[1,1]^2)$. This is possible because $(K_{F_5}+C).C = -2$, and
$\mathcal{L}^\prime.C = 12d-38m+a-4 > -2$ if
\[ a > 2(19m-6d)+2. \]
Under this assumption we only have to show
\[ H^1(F_5, \mathcal{L}^\prime \otimes \mathcal{O}_{F_5}(-C)), \]
and this can be done on $\widetilde{F}$ as above, after using Prop.~\ref{negexc-prop} and assuming the same inequalities.

\item $H^1(P_5, \mathcal{L}_{5|P_5}) = 0$ and 
$H^1(P_5, \mathcal{L}_{5|P_5} \otimes \mathcal{O}_{P_5}(-W_9)) = 0$, for the intersection
with $W_9$: Since $\mathcal{L}_{5|P_5}$ is isomorphic to
\[ \mathcal{L}_{4|P_4}([\frac{19}{2}m-3d,\frac{19}{2}m-3d]^8) \cong 
   \mathcal{L}(a-2(19m-6d);0^4,[0,0]^2,[\frac{19}{2}m-3d,\frac{19}{2}m-3d]^8), \]
we can forget all points with multiplicity $0$ and work with the line bundle 
$\widetilde{\mathcal{L}} := \mathcal{L}(a-2(19m-6d);[\frac{19}{2}m-3d,\frac{19}{2}m-3d]^8)$ on 
$\widetilde{\mathbb{P}} := \mathbb{P}^2([p_1,q_1], \ldots, [p_8,q_8])$. Here $p_1, \ldots, p_4$ lie on the strict transform of a 
conic $C_1$ in $\mathcal{L}(2;[1,1]^4,[0,0]^4)$, whereas $p_5, \ldots, p_8$ lie on the strict transform of a 
conic $C_2$ in $\mathcal{L}(2;[0,0]^4,[1,1]^4)$. These two conics intersect in $4$ points distinct from any point blown up on
$\widetilde{\mathbb{P}}$. The infinitely near points $q_1, \ldots, q_8$ are tangent to $C_1$ resp. $C_2$.

\noindent Set $C = C_1 \cup C_2$. Then:
\[ \widetilde{\mathcal{L}}_{|\widetilde{\mathbb{P}}} \otimes \mathcal{O}_{\widetilde{\mathbb{P}}}(-iC) \cong
   \mathcal{L}(a-2(19m-6d)-4i;[\frac{19}{2}m-3d-i,\frac{19}{2}m-3d-i]^8) \]
and
\begin{eqnarray*}
\left[ \widetilde{\mathcal{L}}_{|\widetilde{\mathbb{P}}} \otimes \mathcal{O}_{\widetilde{\mathbb{P}}}(-iC) \right].C_1 & = & 
2 \cdot (a-2(19m-6d)) - 8i - 8 \cdot (\frac{19}{2}m-3d-i) \\
 & = & 2a - 8(\frac{19}{2}m-3d-i) = 
           \left[ \widetilde{\mathcal{L}}_{|\widetilde{\mathbb{P}}} \otimes \mathcal{O}_{\widetilde{\mathbb{P}}}(-iC) \right].C_2.
\end{eqnarray*}
On the other hand,
$\left[ K_{\widetilde{\mathbb{P}}} \otimes  \mathcal{O}_{\widetilde{\mathbb{P}}}(C) \right].C_i = \mathcal{L}(1;[0,0]^8).C_i = 2$. 
Consequently we can apply Theorem~\ref{curvevan-crit} iteratively for $i=0,1 \ldots, \frac{19}{2}m-3d$, if 
$a > 4(19m-6d)+1$. Finally
\[ \widetilde{\mathcal{L}}_{|\widetilde{\mathbb{P}}} \otimes \mathcal{O}_{\widetilde{\mathbb{P}}}(-(\frac{19}{2}m-3d)C) \cong
   \mathcal{L}(a - 4(19m-6d);[0,0]^8) \]
is non-special. 

\noindent For the surjectivity on $V_{10} \cap W_9$ the intersection curves on $P_5$ add up to a section of 
$\mathcal{L}(5;1^4,[1,1]^2,[1,1]^8)$. Prop.~\ref{negexc-prop} tells us that it will be enough to show
\[ H^1(\widetilde{\mathbb{P}}, \widetilde{\mathcal{L}}_{|\widetilde{\mathbb{P}}} \otimes \mathcal{L}(-5;[0,0]^8)) = 0. \]
Repeating the calculations above shows that we can still apply Theorem~\ref{curvevan-crit} iteratively for 
$i=0,1 \ldots, \frac{19}{2}m-3d$ and obtain a non-special line bundle if $a > 4(19m-6d)+5$. 
\end{enumerate}

\subsubsection{Bounds}

We can apply the Gluing Lemma~\ref{glue-lem} if the following inequalities are satisfied:
\[ d > \sqrt{10}m,\ a < 5m - \frac{3}{2}d,\ a > 4(19m-6d) + 2,\ 19m - 6d \geq 0,\ d < \frac{136}{43}m,\ 
   a \leq \frac{69}{2}d-109m. \]
Since $\frac{69}{2}d-109m < 5m - \frac{3}{2}d \Leftrightarrow 36d < 114m$, $\frac{136}{43} < \frac{19}{6}$ and
\[ 4(19m-6d) + 2 < \frac{69}{2}d-109m \Leftrightarrow 185m+2 < \frac{117}{2}d  \Leftrightarrow 
   \frac{370}{117}m + \frac{4}{117} < d, \]
we can apply Theorem~\ref{Ineq-thm}
with $\mu = \frac{330}{117}$.
\begin{prop}
The multi-point Seshadri constant of $10$ points in general position is bounded from below by
\[ \epsilon(\mathbb{P}^2, \mathcal{O}_{\mathbb{P}^2}(1); p_1, \ldots, p_{10}) \geq \frac{117}{330}. \]
\end{prop}

\section{Algorithmic aspects}  \label{Alg-sec}

\noindent We do not stop with the Fifth Degeneration because a new idea is needed but because the amount of data we need
to keep track of becomes unmanagable by hand. We illustrate this by identifying the next candidates of curves to throw.

\subsection{The next degeneration: Curves to throw}

Assume that $\sqrt{10}m < d < \frac{370}{117}m$. Then $\frac{69}{2}d-109m < 4(19m-6d)$ and we cannot choose $a$ such that
\[ 4(19m-6d) < a < \frac{69}{2}d-109m. \]
We also have $4(19m-6d) < 13d-41m \Leftrightarrow 117m<37d$ because $\frac{117}{37}<\sqrt{10}$. So let us assume from
now on
\[ \frac{69}{2}d-109m < 4(19m-6d) < a < 13d-41m. \] 
Furthermore we modify the line bundle $\mathcal{L}_5$ to
\[ \mathcal{L}^\prime := \mathcal{L}_5 \otimes \mathcal{O}_{\mathcal{X}_5}((\frac{69}{2}d-109m-a) \sum T_{1,i}^{(3,5)}), \]
for the reasons discussed in Remark~\ref{mod-rem}. Using 
\[ \mathcal{O}_{F_5}(T_{1,1}^{(3,5)}) \cong \mathcal{L}(0;0,0^6,[-1,1],[0,0])\ \mathrm{and\ } 
\mathcal{O}_{F_5}(T_{1,2}^{(3,5)}) \cong \mathcal{L}(0;0,0^6,[0,0],[-1,1]) \] 
we obtain
\[ \mathcal{L}^\prime \cong \mathcal{L}(76d-240m-3a; (13d-41m-a)^6, [0, 69d-218m-2a]^2). \]
Consequently, the two $(-1)$-curves $E_{5,1}$ in $\mathcal{L}(0;0,0^6,[0,-1],[0,0])$ and 
$E_{5,2}$ in $\mathcal{L}(0;0,0^4,[0,0],[0,-1])$ are our next candidates for curves to throw. We can throw them simultaneously
because they do not intersect on $F_5$.

\noindent The intersection curves of $F_5$ with the other components add up to a section of 
$\mathcal{L}(15;1,3^6,[6,7],[6,7])$. Hence $E_{5,1}$ resp. $E_{5,2}$ intersect the other components in $7$ points (if they are
different). So we must perform two $7$-throws.

\subsection{Non-termination of algorithm}

The increasing amount of bookkeeping might be tedious to cope with by hand but would not pose any difficulties for a
computer, at least in the next steps. On the other hand it is also interesting to prove general statements which ensure that
the algorithm never terminates. In the following we specify and shortly discuss some issues related to that aim.

\subsubsection{Existence of curves to be thrown}

If a line bundle $\mathcal{L}$ is special on $\mathbb{P}^2$ blown up in several points in general position, the existence of a
$(-1)$-curve intersecting $\mathcal{L}$ sufficiently negative is predicted by the Harbourne-Hirschowitz Conjecture. But in the
degenerations constructed above we already observe components of the central fiber which are isomorphic to 
$\mathbb{P}^2$ blown up in points in rather special positions. In particular, we must deal with omnipresent infinitely
near points. 

\noindent Nevertheless we always found curves to throw among the $(-1)$-curves of the exceptional configuration in which
the restriction of $\mathcal{L}$ is described. A better understanding of why they exist would be desirable.

\subsubsection{Transversal intersections}

The curves to throw should intersect the other components of the central fiber transversally. Otherwise, the Throwing 
Construction~\ref{throw-const} is not applicable, or must be extended to a much more complicated situation. 

\noindent In the above degenerations transversality is always a consequence of sufficiently general position of blown up
points. But when continuing the algorithm more intricate configurations might occur. 

\subsubsection{Modification of degenerated line bundle}

We modified the line bundle on the central fiber in the First, Third and Fifth Degeneration, and we will also need to do it in a
possible Sixth Degeneration, see the section before. The modifications can always be justified as in Remark~\ref{mod-rem},
and use analogous components.

\subsubsection{Position of points} 

Even if the blown up points on a component of the central fiber are not in general position
they should not lie in a too special configuration. In the above degenerations the necessary generality can always be deduced
from the general position of the $10$ points blown up in the beginning.

\subsubsection{Verifying non-specialty}

In all cases in which Harbourne's Criterion~\ref{Har-crit} does not work we were able to simplify the situation with
Criterion~\ref{curvevan-crit}. This was possible because lots of the blown up points in the considered components of the
central fiber lie on simple curves. This is inherent to the algorithm, because new points always occur on intersection
curves with other components. 

\noindent When we applied Harbourne's Criterion~\ref{Har-crit} we did not motivate  the choice of Cremona transformations to
standardize the line bundle. Harbourne~\cite{Har85} developped an algorithm for standardization, for fixed degree and 
multiplicities. But in our case, degree and multiplicities depend on the parameters $d, m, a$, and which Cremona
transformations lead to a standardized line bundle, depends on linear inequalities between these parameters. On the other
hand these linear inequalities are exactly what we want to find. 

\noindent Therefore, a more systematic approach tries different inequalities, their effect on the standardization, and finally
decide which set of linear inequalities gives the best bound in the end. But this is very tedious. 

\subsection{Future prospects}

Besides trying to find bounds for the Seshadri constant of $10$ points on $\mathbb{P}^2$ we could also start the algorithm to
find bounds for the Seshadri constant of $11, 12, \ldots$ points in general position on $\mathbb{P}^2$. But after some steps
we will encounter the difficulties described above in all these cases. 

\noindent On the other hand overcoming these difficulties only requires careful bookkeeping and systematic trial-and-error.
These are tasks perfectly fit to a computer. So if we want to find new bounds for Seshadri constants, we should first program 
a package of tools which allow us to navigate through the data accumulated by the algorithm, without too much
effort.  



\end{document}